\documentclass[11pt,reqno]{amsart}
\usepackage{fullpage}

\newcommand{\Ae}{\mathfrak{ae}}
\renewcommand{\le}{\leqslant}
\renewcommand{\ge}{\geqslant}

\renewcommand{\setminus}{\smallsetminus}

\renewcommand{\subset}{\subseteq}
\renewcommand{\supset}{\supseteq}
\newcommand{\n}{\{1,\ldots,n\}}

\newcommand{\B}{\mathscr{B}}
\newcommand{\e}{\varepsilon}
\newcommand{\Lip}{\mathrm{Lip}}
\newcommand{\conv}{\mathrm{conv}}

\usepackage{amsmath,amsthm,amsfonts,amssymb,verbatim}
\usepackage{paralist}
\usepackage{enumerate}
\usepackage[all]{xy}
\usepackage[mathscr]{euscript}
\usepackage{color}
\usepackage{upgreek}
\renewcommand{\gamma}{\upgamma}
\renewcommand{\1}{\mathbf 1}
\usepackage{graphicx,verbatim}
\usepackage{paralist}
\usepackage[breaklinks,pdfstartview=FitH]{hyperref}

\newcommand{\alert}[1]{\textbf{\color{red}
[[[#1]]]}\marginpar{\textbf{\color{red}**}}\typeout{ALERT:
\the\inputlineno: #1}}
\newcommand{\vol}{{\mathrm{vol}}}

\newcommand{\diam}{{\rm diam}}

\newcommand{\N}{\mathbb{N}}

\newcommand{\R}{\mathbb{R}}

\newcommand{\F}{\mathbb{F}}

\newcommand{\mommit}[1]{}
\newcommand{\namedref}[2]{\hyperref[#2]{#1~\ref*{#2}}}

\theoremstyle{plain}
\newtheorem{theorem}{Theorem}
\newtheorem{lemma}[theorem]{Lemma}
\newtheorem{corollary}[theorem]{Corollary}
\newtheorem{remark}[theorem]{Remark}

\theoremstyle{definition}

\newtheorem{question}[theorem]{Question}

\newenvironment{RETHM}[2]{\trivlist \item[\hskip 10pt\hskip\labelsep{\bf
#1\hskip 5pt\relax\ref{#2}.}]\it}{\endtrivlist}
\newcommand{\rethm}[1]{\begin{RETHM}{Theorem}{#1}}
\newcommand{\repro}[1]{\begin{RETHM}{Proposition}{#1}}
\newcommand{\relem}[1]{\begin{RETHM}{Lemma}{#1}}
\newcommand{\recor}[1]{\begin{RETHM}{Corollary}{#1}}
\newcommand{\erethm}{\end{RETHM}}
\renewcommand{\epsilon}{\varepsilon}

\newcommand{\eqdef}{\stackrel{\mathrm{def}}{=}}

\title[Lipschitz extension from finite subsets]{On Lipschitz extension from finite subsets}

\author{Assaf Naor}
\address{Mathematics Department\\ Princeton University\\ Fine Hall, Washington Road, Princeton, NJ 08544-1000, USA.}
\email{naor@math.princeton.edu}
\thanks{A.~N. was supported in part by the NSF, the BSF, the Packard Foundation and the Simons Foundation.}
\author{Yuval Rabani}
\address{The Rachel and Selim Benin School of Computer Science and Engineering, Jerusalem 91904,
Israel.}
\email{yrabani@cs.huji.ac.il}
\thanks{Y.~R. was supported in part by the ISF, the BSF, and by the Israeli Center of Excellence on Algorithms.}

\begin{document}
\maketitle


\begin{abstract}
We prove that for every $n\in \N$ there exists a metric space $(X,d_X)$, an $n$-point subset $S\subset X$, a Banach space $(Z,\|\cdot\|_Z)$ and a $1$-Lipschitz function $f:S\to Z$ such that the Lipschitz constant of every function $F:X\to Z$ that extends $f$ is  at least a constant multiple of $\sqrt{\log n}$. This improves a bound of Johnson and Lindenstrauss~\cite{JL84}. We also obtain the following quantitative counterpart to a classical extension theorem of Minty~\cite{Min70}. For every $\alpha\in (1/2,1]$ and  $n\in \N$ there exists a metric space $(X,d_X)$, an $n$-point subset $S\subset X$ and a function $f:S\to \ell_2$ that is $\alpha$-H\"older with constant~$1$, yet the $\alpha$-H\"older constant of any $F:X\to \ell_2$ that extends $f$ satisfies
$$
\|F\|_{\Lip(\alpha)}\gtrsim (\log n)^{\frac{2\alpha-1}{4\alpha}}+\left(\frac{\log n}{\log\log n}\right)^{\alpha^2-\frac12}.
$$
We formulate a conjecture whose positive solution would strengthen Ball's nonlinear Maurey extension theorem~\cite{Bal92}, serving as a far-reaching nonlinear version of a theorem of K\"onig, Retherford and Tomczak-Jaegermann~\cite{KRT80}. We explain how this conjecture would imply as special cases answers to longstanding open questions of Johnson and Lindenstrauss~\cite{JL84} and Kalton~\cite{Kal04}.
\end{abstract}


\section{Introduction}

Given two metric spaces $(X,d_X)$ and $(Z,d_Z)$, the Lipschitz constant of a mapping $f:X\to Z$ will be denoted below by $\|f\|_{\Lip}$. For every subset $S\subset X$, let $e(X,S,Z)$ denote the infimum over those $K\in (0,\infty]$ with the property that for every $f:S\to Z$ there exists $F:X\to Z$ with $\|F\|_{\Lip}\le K\|f\|_{\Lip}$ and whose restriction to $S$ satisfies $F|_S=f$. In its most general form, the {\em Lipschitz extension problem} asks for estimates on the quantity $e(X,S,Z)$.

It is of great interest to obtain geometric conditions on the metric spaces $(X,d_X)$ and $(Z,d_Z)$ ensuring that $e(X,S,Z)<\infty$ for every $S\subset X$. Formally, define $e(X,Z)$ to be the supremum of $e(X,S,Y)$ over all $S\subset X$. When $e(X,Z)=\infty$, it is natural to refine the Lipschitz extension problem by taking $n\in \N$ and defining $e_n(X,Z)$ to be the supremum of $e(X,S,Z)$ over all $S\subset X$ of cardinality at most $n$, and asking for the asymptotic behavior of $e_n(X,Z)$ as $n\to \infty$. Another natural quantitative refinement of the parameter $e(X,Z)$ is to fix $\e\in (0,1]$ and define $e_\e(X,Z)$ to be the supremum of $e(X,S,Z)$ over all $S\subset X$ that are $\e$-discrete in the sense that $d_X(x,y)\ge \e\cdot \diam(S)$ for every distinct $x,y\in S$, and asking for the asymptotic behavior of $e_\e(X,Z)$ as $\e\to 0$.

Denote the supremum of $e_n(X,Z)$ over all metric spaces $(X,d_X)$ and all Banach spaces $(Z,\|\cdot\|_Z)$ by $\Ae(n)$. Here $``\mathfrak{ae}"$ stands for ``absolute extendability," where we are following the notation and terminology that was introduced in~\cite{LN04}. Thus, if $\Ae(n)<K$ then for {\em every} $n$-point metric space $(M,d_M)$, {\em every} Banach-space valued $1$-Lipschitz function defined on $M$ can be extended to {\em any} metric space that (isometrically) contains $(M,d_M)$ so that the Lipschitz constant of the extended function is at most $K$. One can similarly define for $\e>0$ the quantity $\Ae(\e)$ by considering the supremum of $e_\e(X,Z)$ over all metric spaces $(X,d_X)$ and all Banach spaces $(Z,\|\cdot\|_Z)$.

The finiteness of $\Ae(n)$ and $\Ae(\e)$ for every $n\in \N$ and $\e\in (0,1]$ is well-known. Johnson, Lindenstrauss and Schechtman proved in~\cite{JLS86} that $\Ae(n)\lesssim \log n$, where here, and in what follows, the notation $A\lesssim B$ (respectively $A\gtrsim B$) stands for $A\le CB$ (respectively $A\ge CB$) for some universal constant $C\in (0,\infty)$. We shall also use the notation $A\asymp B$ when $A\lesssim B$ and $A\gtrsim B$. The above upper bound of~\cite{JLS86} on $\Ae(n)$ was more recently improved~\cite{LN04,LN05} to $\Ae(n)\lesssim (\log n)/\log\log n$, which is the best-known upper bound on $\Ae(n)$ to date. Determining the asymptotic behavior of $\Ae(n)$ as $n\to \infty$ remains a major open problem. The previously best-known lower bound on $\Ae(n)$ is $\Ae(n)\gtrsim \sqrt{(\log n)/\log \log n}$, obtained over thirty years ago by Johnson and Lindenstrauss~\cite{JL84} (see also~\cite{MM10} for a variant of the construction of Johnson and Lindenstrauss that also yields the same lower bound). Here we obtain the following improvement.

\begin{theorem}\label{thm:main}
For every $n\in \N$ we have $\Ae(n)\gtrsim \sqrt{\log n}$.
\end{theorem}
Thus, the best-known bounds for $\Ae(n)$ are now
\begin{equation}\label{eq:best known}
\sqrt{\log n}\lesssim \Ae(n)\lesssim \frac{\log n}{\log\log n}.
\end{equation}
It would be very interesting to improve any of the bounds in~\eqref{eq:best known}, the ultimate goal being to determine the rate at which $\Ae(n)$ tends to $\infty$ as $n\to \infty$.

In~\cite{JL84} it was shown that $\Ae(\e)\lesssim 1/\e$, and a different and very simple proof of this fact was given in~\cite{JLS86}. In~\cite{MN13} it was proved that $\Ae(\e)\gtrsim 1/\e$, via a quantitative refinement of a beautiful argument of Kalton~\cite{Kal12}. Thus $\Ae(\e)\asymp 1/\e$. The example that leads to Theorem~\ref{thm:main} also yields a new proof that $\Ae(\e)\gtrsim 1/\e$ which some may find to be somewhat simpler than the proof in~\cite{MN13}; see Remark~\ref{rem:aspect} below. We note, however, that the example of~\cite{MN13} has  important properties (see the discussion in Section 1.6 of~\cite{MN13}) that Theorem~\ref{thm:main} does not imply.

The value of Theorem~\ref{thm:main} does not stem only  from the fact that it yields an asymptotic improvement over the best known bound, but rather because this improvement relies on a conceptually different approach than that of~\cite{JL84}. In fact, as emphasized explicitly in~\cite{JL84}, the difficulty to improve the lower bound on $\Ae(n)$ to a constant multiple of $\sqrt{\log n}$ is not purely technical, and the approach of~\cite{JL84} inherently cannot  yield a lower bound that is better than $o(\sqrt{\log n})$.

To explain the above assertions, we need to briefly sketch the argument of~\cite{JL84}. Fix $k\in \N$, $L>0$ and $\e\in (0,1/2)$. Suppose that $(X,\|\cdot\|_X)$ is a Banach space and $Y\subset X$ is a $k$-dimensional linear subspace of $X$. Let $\mathscr{N}_\e$ be an $\e$-net in the unit sphere of $Y$ and denote $S_\e=\mathscr{N}_\e\cup\{0\}$. Taking $f:S_\e\to S_\e\subset Y$ to be the identity mapping, it is proved in~\cite{JL84} that if $F:X\to Y$ is an $L$-Lipschitz extension of $f$ and $\e$ is sufficiently small then there exists a linear projection $P:X\to Y$ with $\|P\|\lesssim L$. Hence, if for some $\Lambda>0$ one knows that $\|P\|\ge \Lambda$ for {\em every} linear projection $P:X\to Y$ then one concludes that $L\gtrsim \Lambda$. Suppose that $\Lambda\gtrsim \sqrt{k}$; since the classical Kadec$'$--Snobar theorem~\cite{KS71} (see also~\cite[\S III.B]{Woj91}) asserts that we always have $\Lambda\le \sqrt{k}$, this lower bound on $\Lambda$ is the best one could hope for. By standard bounds on the size of $\e$-nets (e.g.~\cite{MS86}), if we set $n=|S_\e|$ then $\log n\asymp  k\log(1/\e)$, so we have $L\gtrsim \sqrt{(\log n)/\log(1/\e)}$. For this strategy to yield a lower bound that is a constant multiple of $\sqrt{\log n}$, one needs to take $\e$ to be a universal constant, but in this case we would have $L\lesssim \Ae(2\e)\lesssim 1/\e=O(1)$. So, in order to get a lower bound on $L$ that tends to $\infty$ with $n$  one {\em must have} $\e=o(1)$ as $k\to \infty$ (indeed, in~\cite{JL84} the choice of $\e$ is $\log(1/\e)\asymp \log k\asymp \log\log n$). This explains why one cannot prove Theorem~\ref{thm:main} via the above strategy of reduction to the nonexistence of linear projections of small norm. Note also that due to the  Kadec$'$--Snobar theorem~\cite{KS71}, this ``linearization" approach cannot yield a lower bound on $\Ae(n)$ that tends to $\infty$ at a rate that is faster than a constant multiple of $\sqrt{\log n}$.

Here we overcome the above obstacle by abandoning entirely the desire to reduce the problem to the nonexistence of linear projections of small norm. We consider a family of metric spaces that arise from (a modification of) expander graphs, and our poorly extendable functions take values in their associated Wasserstein $1$ spaces. The proof  relies on direct geometric considerations, in particular the use of Banach space-valued Poincar\'e inequalities (see Lemma~\ref{lem:poincare into wasserstein} below), rather than a reduction to linear mappings. Our approach is inspired by  a certain algorithmic clustering problem for graphs, specifically by leveraging the difference between two relaxations  of the corresponding combinatorial optimization problem. This motivation is explained in Section~\ref{sec:algorithmic} below, though the proof of Theorem~\ref{thm:main} (in Section~\ref{sec:proof main}) is direct and does not rely on any algorithmic background.

\subsection{Hilbert space-valued H\"older functions}\label{sec:holder intro}

Suppose that $(X,d_X)$ and $(Z,d_Z)$ are metric spaces and $\alpha\in (0,1]$. We shall denote below the $\alpha$-H\"older constant of a mapping $f:X\to Z$ by $\|f\|_{\Lip(\alpha)}$.   Thus $\|f\|_{\Lip(\alpha)}$ is the infimum over those $L\in (0,\infty]$ for which every $x,y\in X$ satisfy  $d_Z(f(x),f(y))\le Ld_X(x,y)^\alpha$. Equivalently,  $\|f\|_{\Lip(\alpha)}$ is the Lipschitz constant of $f$ when it is viewed as a mapping between from the metric space $(X,d_X^\alpha)$ to the metric space $(Z,d_Z)$. Denote
$$
e^\alpha(X,Z)\eqdef e\big((X,d_X^\alpha),(Z,d_Z)\big)\qquad\mathrm{and}\qquad \forall\, n\in \N,\qquad e^\alpha_n(X,Z)\eqdef e_n\big((X,d_X^\alpha),(Z,d_Z)\big).
$$
Thus, $e^\alpha(X,Z)$ is the infimum over those $K\in [1,\infty]$ such that for every $S\subset X$ and every $f:S\to Z$ there exists $F:X\to Z$ with $F|_S=f$ and $\|F\|_{\Lip(\alpha)}\le K\|f\|_{\Lip(\alpha)}$, and analogously for $e^\alpha_n(X,Z)$.

A classical theorem of Minty~\cite{Min70}  asserts that if $H$ is a Hilbert space then $e^\alpha(X,H)=1$ for every metric space $(X,d_X)$ and $\alpha\in (0,1/2]$ (equivalently, $e_n^\alpha(\ell_\infty,\ell_2)=1$ for all $n\in \N$). Minty's theorem fails when $\alpha\in (1/2,1]$ (see~\cite{HW71}), but understanding what happens when $\alpha\in (1/2,1]$ remains a mystery. Specifically, Kalton conjectured in~\cite{Kal04} that for every $\alpha\in [1/2,1]$ and $n\in \N$,
\begin{equation}\label{eq:kalton conj}
e^\alpha(\ell_\infty,\ell_2^n)\lesssim n^{\alpha-\frac12}.
\end{equation}
This conjecture appears as Problem~11.3 in~\cite{Kal04}\footnote{We note in passing that there is a misprint in~\cite[Problem~11.3]{Kal04}: in the two displayed equations that appear there the exponent $\alpha$ erroneously appears in the left hand side of the inequality rather than in its right hand side.}. See the discussion immediately following Problem~11.3 in~\cite{Kal04} for an interesting geometric application that~\eqref{eq:kalton conj} would imply. The estimate~\eqref{eq:kalton conj} holds true when $\alpha=1/2$ due to Minty's theorem, and also when $\alpha=1$ by the fact that $\ell_2^n$ is $\sqrt{n}$-isomorphic to $\ell_\infty^n$, combined with an application of the nonlinear Hahn--Banach theorem (see e.g.~\cite[Lemma~1.1]{BL00}). Thus~\eqref{eq:kalton conj} is a natural conjectural interpolation between Minty's theorem and the nonlinear Hahn-Banach theorem.

Kalton's conjecture~\eqref{eq:kalton conj} remains an interesting open problem, and in Section~\ref{sec:conjecture} we formulate a conjecture whose validity would imply the validity of~\eqref{eq:kalton conj}. In the reverse direction, to the best of our knowledge its isn't known whether~\eqref{eq:kalton conj} would be sharp. We therefore ask the following question.
\begin{question}\label{Q:kalton sharp}
Is it true that $e^\alpha(\ell_\infty,\ell_2^n)\gtrsim n^{\alpha-\frac12}$  for every $\alpha \in (1/2,1)$?
\end{question}
We also ask the following variant of Question~\ref{Q:kalton sharp} for the parameter $e_n^\alpha(\ell_\infty,\ell_2)$.
\begin{question}\label{Q:kalton sharp n}
Is it true that $e_n^\alpha(\ell_\infty,\ell_2)\gtrsim (\log n)^{\alpha-\frac12}$  for every $\alpha \in (1/2,1)$?
\end{question}
By mimicking an argument of~\cite{JL84} one sees that a positive answer to Question~\ref{Q:kalton sharp n} would imply  a positive answer to Question~\ref{Q:kalton sharp}. Indeed, for every metric space $(X,d_X)$ and $\alpha\in (0,1]$ we have
\begin{equation}\label{eq:dimension reduction}
e^\alpha_{2^{7n}}(X,\ell_2)\le 3 e^\alpha(\ell_\infty,\ell_2^n).
\end{equation}
(The choice of constants in~\eqref{eq:dimension reduction} is somewhat arbitrary; see below.) Since~\eqref{eq:dimension reduction} is not stated explicitly in~\cite{JL84}, we shall now briefly explain how it is proved. Fix $S\subset X$ with $|S|\le 2^{7n}$ and $f:S\to \ell_2$. Since $f(S)$ is a subset of a Hilbert space of cardinality at most $2^{7n}$, by the Johnson--Lindenstrauss dimensionality reduction lemma~\cite{JL84} there exists a mapping $g:f(S)\to \ell_2^n$ such that
\begin{equation}\label{eq:JL quote}
\forall\, a,b\in f(S),\qquad \|a-b\|_2\le \|g(a)-g(b)\|_2\le 3\|a-b\|_2.
\end{equation}
The specific parameters in the Johnson--Lindenstrauss lemma that we take here are valid due to the bounds in~\cite[Theorem~1.1]{Ach03}.  By the Kirszbraun extension theorem~\cite{Kirsz34},
\begin{equation}\label{eq:use kirsz}
\exists\, \widetilde{g^{-1}}:\ell_2^n\to \ell_2\ \ \mathrm{such\ \ that}\ \ \left.\widetilde{g^{-1}}\right|_{g\circ f(S)}=g^{-1}\ \ \mathrm{and}\ \ \left\|\widetilde{g^{-1}}\right\|_{\Lip}\le \left\|g^{-1}\right\|_{\Lip}\le 1,
\end{equation}
where the last step of~\eqref{eq:use kirsz} uses the left hand inequality in~\eqref{eq:JL quote}. Also, by the definition of $e^\alpha(\ell_\infty,\ell_2^n)$,
\begin{equation}\label{eq:use alpha holder assumption}
\exists\, \widetilde{g\circ f}:X\to \ell_2^n\ \ \mathrm{such\ \ that}\ \ \left.\widetilde{g\circ f}\right|_{S}=g\circ f\ \ \mathrm{and}\ \ \left\|\widetilde{g\circ f}\right\|_{\Lip(\alpha)}\le e^\alpha(\ell_\infty,\ell_2^n)\left\|g\circ f\right\|_{\Lip(\alpha)}.
\end{equation}
Thus, the following diagram commutes.
$$
\xymatrix @C=7.7pc {
X \ar@/^2.7pc/[rr]^{\widetilde{g\circ f}} \ar[r]^{\widetilde{g^{-1}}\circ \widetilde{g\circ f}} &
 \ell_2
&   \ar[l]_{\widetilde{g^{-1}}}
\ell_2^{n}\\
\strut   S \ar@{^{(}->}[u]^{\subset} \ar[r]^{f} & \ar@{^{(}->}[u]^{\subset} f(S)\ar@/^1.7pc/[r]^{ g }  & \ar@/^1.7pc/[l]^{ g^{-1} }
\ar@{^{(}->}[u]^{\subset}  g\circ f (S)
}
$$
The mapping $\widetilde{g^{-1}}\circ \widetilde{g\circ f}:X\to \ell_2$ therefore extends $f$ and satisfies
\begin{multline}\label{eq:dimension reduced}
\left\|\widetilde{g^{-1}}\circ \widetilde{g\circ f}\right\|_{\Lip(\alpha)}\le \left\|\widetilde{g^{-1}}\right\|_{\Lip}\cdot \left\|\widetilde{g\circ f}\right\|_{\Lip(\alpha)}\stackrel{\eqref{eq:use kirsz}\wedge\eqref{eq:use alpha holder assumption}}{\le} e^\alpha(\ell_\infty,\ell_2^n)\left\|g\circ f\right\|_{\Lip(\alpha)}\\\le e^\alpha(\ell_\infty,\ell_2^n)\|g\|_{\Lip} \|f\|_{\Lip(\alpha)}\le
3e^\alpha(\ell_\infty,\ell_2^n)\|f\|_{\Lip(\alpha)},
\end{multline}
where the last step of~\eqref{eq:dimension reduced} uses the right hand inequality in~\eqref{eq:JL quote}. Clearly~\eqref{eq:dimension reduced} implies~\eqref{eq:dimension reduction}.

The link between Question~\ref{Q:kalton sharp} and Theorem~\ref{thm:main} is twofold. Firstly, Question~\ref{Q:kalton sharp} presents another situation in which the linearization strategy of~\cite{JL84} for proving extension lower bounds seems to be insufficient, because bounded linear maps are Lipschitz rather than $\alpha$-H\"older for  $\alpha\in (0,1)$. Nevertheless, in Section~\ref{sec:linearization} we do show how to use a linearization procedure in the spirit of~\cite{JL84} to bound $e^\alpha(\ell_\infty,\ell_2^n)$ from below. However, our approach requires a substantial modification of the argument of~\cite{JL84} and at present we do not see how it could yield a nontrivial result for the entire range $\alpha\in (1/2,1]$ (the argument of Section~\ref{sec:linearization} yields nontrivial bounds only when $\alpha\ge 1/\sqrt{2}$). The second link between Question~\ref{Q:kalton sharp} and Theorem~\ref{thm:main} is that the bounds that we obtain in Theorem~\ref{thm:holder} below are based on a modification of a construction of~\cite{JLS86} which is part of a general family of constructions that were used in related contexts also in~\cite{Lan99,CKR04}, and to which the example that underlies Theorem~\ref{thm:main} belongs as well. We elaborate further on this in Remark~\ref{rem:general unions} below.

Here we obtain the best known lower bounds in the context of Question~\ref{Q:kalton sharp} and Question~\ref{Q:kalton sharp n}.

\begin{theorem}\label{thm:holder} For every $\alpha\in (1/2,1]$ and every $n\in \N$ we have
\begin{equation}\label{eq:lower alpha}
e^\alpha(\ell_\infty,\ell_2^n)\gtrsim n^{\frac{2\alpha-1}{4\alpha}}+n^{\alpha^2-\frac12}\qquad \mathrm{and}\qquad e_n^\alpha(\ell_\infty,\ell_2)\gtrsim (\log n)^{\frac{2\alpha-1}{4\alpha}}+\left(\frac{\log n}{\log\log n}\right)^{\alpha^2-\frac12}.
\end{equation}
\end{theorem}
Importantly, the lower bounds in~\eqref{eq:lower alpha} tend to $\infty$ with $n$ for every $\alpha\in (1/2,1]$. A positive answer to Question~\ref{Q:kalton sharp} (respectively, Question~\ref{Q:kalton sharp n}) for $\alpha=\frac12 +\e$ would yield the lower bound $e^\alpha(\ell_\infty,\ell_2^n)\gtrsim n^\e$ (respectively, $e_n^\alpha(\ell_\infty,\ell_2)\gtrsim (\log n)^\e$), while Theorem~\ref{thm:holder} implies that
$$
e^{\frac12+\e}(\ell_\infty,\ell_2^n)\gtrsim n^{\e-2\e^2}\qquad\mathrm{and}\qquad e_n^{\frac12+\e}(\ell_\infty,\ell_2)\gtrsim (\log n)^{\e-2\e^2},
$$
i.e., the  exponents of Theorem~\ref{thm:holder} match what Kalton's conjecture predicts up to lower order terms as $\alpha\to 1/2^+$. The exponent $\alpha^2-1/2$ in the second summands in~\eqref{eq:lower alpha} becomes positive only when $\alpha\ge 1/\sqrt{2}$, and it becomes greater than the exponent of the first summands in~\eqref{eq:lower alpha}, namely $(2\alpha-1)/(4\alpha)$, only when $\alpha\in (\alpha_0,1]$, where $\alpha_0=0.837...$ is the largest root of the polynomial $4x^3-4x+1$. As $\alpha$ tends to $1$ the exponent $\alpha^2-1/2$ tends to $1/2$, so the bounds in~\eqref{eq:lower alpha} yield an alternative possible interpolation between Minty's theorem and the nonlinear Hahn--Banach theorem. We do not believe that this interpolation is sharp and it seems more likely that Kalton's conjecture, and correspondingly Question~\ref{Q:kalton sharp} and Question~\ref{Q:kalton sharp n}, have positive answers. It may very well be the case that the example that we analyse in  Section~\ref{sec:linearization} is itself an example that yields a positive answer to Question~\ref{Q:kalton sharp}, but at present we do not know how to prove this.

\subsection{A conjectural strengthening of Ball's extension theorem}\label{sec:conjecture} A longstanding open problem posed by Johnson and Lindenstrauss in~\cite[Problem~1]{JL84} asks whether or not for every $p\in (1,2)$ and every $n\in \N$ we have
\begin{equation}\label{eq:JL Q}
e(\ell_p,\ell_2^n)\lesssim_p n^{\frac{1}{p}-\frac12}.
\end{equation}
Here, and in what follows, the notation $A\lesssim_p B$ stands for $A\le C(p)B$, where $C(p)\in (0,\infty)$ is allowed to depend only on $p$. The notations $A\gtrsim _p B$ and $A\asymp_p B$ are defined analogously.

Johnson and Lindenstrauss were motivated to ask whether~\eqref{eq:JL Q} holds true by a Lipschitz extension theorem of Marcus and Pisier~\cite{MP84}, which states that for every $p\in (1,2)$ and $n\in \N$ we have
\begin{equation}\label{eq:quote MP}
e_n(\ell_p,\ell_2)\lesssim_p (\log n)^{\frac{1}{p}-\frac12}.
\end{equation}
Specifically, the validity of~\eqref{eq:JL Q} would yield a new proof of the Marcus--Pisier theorem~\eqref{eq:quote MP} through an application of the Johnson--Lindenstrauss dimensionality reduction lemma (we recalled how such arguments are carried out in the proof of~\eqref{eq:dimension reduction} above). Thus, the validity of~\eqref{eq:JL Q} would yield an illuminating new perspective on the work of Marcus and Pisier~\cite{MP84}, who proved that~\eqref{eq:quote MP} holds true via an entirely different argument.

Our goal in this section is to formulate a conjecture that implies both Kalton's conjecture on the validity of~\eqref{eq:kalton conj} and the Johnson--Lindenstrauss conjecture on the validity of~\eqref{eq:JL Q}, in addition to a wealth of yet unknown Lipschitz extension results. In order to do so, we need to quickly recall Ball's work~\cite{Bal92} (itself addressing another open question that was posed by Johnson and Lindenstrauss in~\cite{JL84}) on the nonlinear version of Maurey's extension theorem~\cite{Mau74}.

For $p,q\in [1,\infty]$, the Rademacher type $p$ constant and Rademacher cotype $q$ constant of a Banach space $(X,\|\cdot\|_X)$, denoted $T_p(X)$ and $C_q(X)$, respectively, are defined to be the infima of those $T,C\in [1,\infty]$ such that for every $n\in \N$ and every $x_1,\ldots,x_n\in X$ we have
$$
\frac{1}{2^n}\sum_{\e\in \{-1,1\}^n} \Big\|\sum_{j=1}^n \e_j x_j\Big\|_X^p\le T^p\sum_{j=1}^n\|x_j\|_X^p\qquad\mathrm{and}\qquad \sum_{j=1}^n\|x_j\|_X^q\le \frac{C^q}{2^n}\sum_{\e\in \{-1,1\}^n} \Big\|\sum_{j=1}^n \e_j x_j\Big\|_X^q.
$$
Maurey's extension theorem~\cite{Mau74} asserts that if $(X,\|\cdot\|_X)$ and $(Z,\|\cdot\|_Z)$ are Banach spaces and $E\subset X$ is a linear subspace then for every linear operator $U:E\to Z$ there exists a linear operator $V:X\to Z$ that extends $U$ and satisfies the operator norm bound $\|V\|_{X\to Z}\lesssim T_2(X)C_2(Z)\|U\|_{E\to Z}$.

In~\cite{Bal92}, Ball discovered a powerful nonlinear version of the Maurey extension theorem. In order to do so, he introduced notions of type and cotype for metric spaces that have since proven to be useful in several contexts beyond their original use for the Lipschitz extension problem. For $p\in (0,\infty)$, the Markov type $p$ constant of a metric space $(X,d_X)$, denoted $M_p(X)$, is defined to be the infimum over those $M\in [1,\infty]$ such that for every $n,t\in \N$, every $n\times n$ symmetric stochastic matrix $A=(a_{ij})$, and every $x_1,\ldots,x_n\in X$ we have
$$
\sum_{i=1}^n\sum_{j=1}^n (A^t)_{ij} d_X(x_i,x_j)^p\le tM^p\sum_{i=1}^n\sum_{j=1}^n a_{ij} d_X(x_i,x_j)^p.
$$
For $q\in (0,\infty)$, the metric Markov cotype $q$ constant of a metric space $(X,d_X)$, denoted $N_p(X)$, is defined to be the infimum over those $N\in (0,\infty]$ such that for every $n,t\in \N$, every $n\times n$ symmetric stochastic matrix $A=(a_{ij})$, and every $x_1,\ldots,x_n\in X$, there exist $y_1,\ldots,y_n\in X$ such that
$$
\sum_{i=1}^n d_X(x_i,y_i)^p+t\sum_{i=1}^n\sum_{j=1}^n a_{ij} d_X(y_i,y_j)^p\le N^p\sum_{i=1}^n\sum_{j=1}^n\frac{1}{t}\sum_{s=1}^t (A^s)_{ij} d_X(x_i,x_j)^p.
$$

Ball's extension theorem~\cite{Bal92} asserts that for every metric space $(X,d_X)$ and every Banach space $(Z,\|\cdot\|_Z)$ we have $$e(X,Z^{**})\lesssim M_2(X) N_2(Z).$$ This formulation is not the strongest known version of Ball's theorem, but it suffices for our purposes; see~\cite{MN13} for a more complete discussion.

\begin{remark}
{\em In the setting of Ball's extension theorem one actually needs to know the Markov type of the {\em complement} of the subset from which one wishes to extend rather than the Markov type of the entire ambient space. Namely, we actually have $$e(X,S,Z^{**})\lesssim M_2(X\setminus S)N_2(Z).$$ This assertion holds true also for the more general version of Ball's extension theorem that was obtained in~\cite{MN13}. The above stronger statement follows effortlessly from an inspection of the proofs in~\cite{Bal92,MN13}, though it has never been stated in the literature; we believe that it is worthwhile to record it here for potential future applications. }
\end{remark}

A beautiful strengthening of Maurey's extension theorem for finite-dimensional targets was discovered by K\"onig, Retherford and Tomczak-Jaegermann~\cite{KRT80}, who proved that for every $p\in [1,2]$, $q\in [2,\infty)$ and $n\in \N$, if $(X,\|\cdot\|_X)$ and $(Z,\|\cdot\|_Z)$ are Banach spaces with $\dim(Z)=n$ and $E\subset X$ is a linear subspace, then for every linear operator $U:E\to Z$ there exists a linear operator $V:X\to Z$ that extends $U$ and satisfies
\begin{equation}\label{eq:KRT}
\|V\|_{X\to Z}\lesssim_{p,q} T_p(X)C_q(Z)n^{\frac{1}{p}-\frac{1}{q}}\|U\|_{E\to Z}.
\end{equation}
See also the expository article of Pisier~\cite{Pis79} for an elegant proof of~\eqref{eq:KRT}. Important earlier special cases of the above result (other than Maurey's extension theorem itself) can be found (via different proofs) in the works of Lewis~\cite{Lew78} and Figiel and Tomczak-Jaegermann~\cite{FT79}.

In light of~\eqref{eq:KRT}, we ask the following question.
\begin{question}\label{Q:finite dim ball}
Suppose that $1\le p\le q<\infty$ and that $(X,d_X)$ is a metric space with Markov type $p$, i.e., $M_p(X)<\infty$. Suppose also that $n\in \N$ and that $(Z,\|\cdot\|_Z)$ is an $n$-dimensional normed space. Is it true that there exists a constant $K=K(M_p(X),N_q(Z),p,q)\in (0,\infty)$, which may depend only on the parameters $M_p(X),N_q(Z),p,q$, such that
$$
e(X,Z)\le K n^{\frac{1}{p}-\frac{1}{q}}\ ?
$$
\end{question}

Since, by a straightforward application of the triangle inequality, every metric space $(X,d_X)$ has Markov type $1$ with $M_1(X)=1$, for every $\alpha \in (0,1]$ the Markov type $1/\alpha$ constant of $(X,d_X^\alpha)$ equals $1$. Also, by~\cite{Bal13,MN13}, for every $q\in (1,\infty)$ we have $N_{\max\{q,2\}}(\ell_q)\lesssim \sqrt{q/(q-1)}$. Hence, a positive answer to Question~\ref{Q:finite dim ball} would imply that for  every $n\in \N$ we have
\begin{equation}\label{eq:implies kalton}
\frac{1}{\max\{q,2\}}\le \alpha\le 1 \implies e^\alpha(\ell_\infty,\ell_q^n)\lesssim_q n^{\alpha-\frac{1}{\max\{q,2\}}}.
\end{equation}
The case $q=2$ of~\eqref{eq:implies kalton} is the same as~\eqref{eq:kalton conj}, i.e., Kalton's conjecture is a special case of Question~\ref{Q:finite dim ball}. By~\cite{Bal92,NPSS06}, for every $p\in [1,\infty)$ we have $M_{\min\{p,2\}}(\ell_p)\lesssim \sqrt{p}$. This implies formally (directly from the definition of Markov type) that for $\alpha\in (0,1]$ the Markov type $\min\{p,2\}/\alpha$ constant of the metric space $(\ell_p,\|x-y\|_p^\alpha)$ is at most a constant multiple of $p^{\alpha/2}$. Hence, a positive answer to Question~\ref{Q:finite dim ball} would also imply that for every $p,q\in (1,\infty)$ and every $n\in \N$ we have
\begin{equation}\label{eq:implies JL}
\frac{\min\{p,2\}}{\max\{q,2\}}\le \alpha\le 1\implies e^\alpha(\ell_p,\ell_q^n)\lesssim_{p,q} n^{\frac{\alpha}{\min\{p,2\}}-\frac{1}{\max\{q,2\}}}.
\end{equation}
The case $p\in (1,2)$,  $q=2$ and $\alpha=1$ of~\eqref{eq:implies JL} is the same as~\eqref{eq:JL Q}, i.e., the above conjecture of Johnson and Lindenstrauss is a special case of Question~\ref{Q:finite dim ball}. The validity of~\eqref{eq:implies JL} would complement the fact that $e^\alpha(\ell_p,\ell_q)<\infty$ for every $\alpha\in (0,\min\{p,2\}/\max\{q,2\}]$, as shown in~\cite{Nao01,NPSS06}.

At present we see several obstacles to adapting the proofs in~\cite{Bal92,MN13} so as to incorporate the finite dimensionality of the target in order to answer Question~\ref{Q:finite dim ball}. We therefore leave Question~\ref{Q:finite dim ball} as an intriguing direction for future research, itself part of the Ribe program (see~\cite{Bal13,Nao12}).

\subsection{Lipschitz extension between $\ell_p$ spaces} Theorem~\ref{thm:main} yields the best known lower bound on $\Ae(n)$ as $n\to \infty$, but there are several cases of special interest where the best known lower bound on $e_n(X,Z)$ is $o(\sqrt{\log n})$. Here we recall the best known bounds when $X=\ell_p$ and $Y=\ell_q$ for $p,q\in [1,\infty]$, as a survey of very  basic questions on Lipschitz extension that remain open.

The case of Hilbert space-valued functions was famously studied by Johnson and Lindenstrauss in~\cite{JL84}, answering a question posed by Marcus and Pisier in~\cite{MP84}. The bounds on $e_n(\ell_\infty,\ell_2)$ that were obtained in~\cite{JL84} are as follows, and they remain the best known bounds to date.
\begin{equation}\label{eq:infty 2}
\frac{\sqrt{\log n}}{\sqrt{\log\log n}}\lesssim e_n(\ell_\infty,\ell_2)\lesssim \sqrt{\log n}.
\end{equation}
In the special case of Hilbert space-valued functions defined on finite subsets of $\ell_p$ for $p\in (1,2)$, the best known bounds are
\begin{equation}\label{eq:MP}
\left(\frac{\log n}{\log\log n}\right)^{\frac{1}{p}-\frac12}\lesssim e_n(\ell_p,\ell_2)\lesssim_p(\log n)^{\frac{1}{p}-\frac12}.
\end{equation}
The right hand inequality in~\eqref{eq:MP} is due to~\cite{MP84} and the left hand inequality in~\eqref{eq:MP} is due to~\cite{JL84}, where it is shown to hold for every $p\in [1,2]$.

As shown by Makarychev and Makarychev in~\cite{MM10}, the parameters $e_n(\ell_\infty,\ell_1)$ and $e_n(\ell_1,\ell_1)$ have a special algorithmic significance. The best known bounds for these quantities are
\begin{equation}\label{eq:infty 1}
\frac{\sqrt{\log n}}{\sqrt{\log\log n}}\lesssim e_n(\ell_\infty,\ell_1)\lesssim \frac{\log n}{\log\log n},
\end{equation}
and
\begin{equation}\label{eq:1 1}
\frac{\sqrt{\log n}}{\log\log n}\lesssim e_n(\ell_1,\ell_1)\lesssim \frac{\log n}{\log\log n}.
\end{equation}
The right hand inequalities in~\eqref{eq:infty 1} and~\eqref{eq:1 1} are a special case of the general upper bound of~\cite{LN05}. The left hand inequalities in~\eqref{eq:infty 1} and~\eqref{eq:1 1} appear in~\cite{MM10}, with the left hand inequality of~\eqref{eq:1 1} being based on the work of Figiel, Johnson and Schechtman~\cite{FJS88} (an asymptotically weaker lower bound in~\eqref{eq:1 1} follows from earlier work of Bourgain~\cite{Bou81}).

By Ball's extension theorem~\cite{Bal92}, combined with the fact~\cite{NPSS06} that $\ell_p$ has Markov type $2$ when $p\in [2,\infty)$, if $1<q\le 2\le p<\infty$  then $e(\ell_p,\ell_q)\le C(p,q)$ for some $C(p,q)\in (0,\infty)$. Here the asymptotic dependence of $C(p,q)$ as $p\to \infty$ or $q\to 1$ remains unknown. In particular, a famous and longstanding open question of Ball~\cite{Bal92} asks whether $e(\ell_2,\ell_1)$ is finite or infinite. The best known bounds on $e_n(\ell_p,\ell_q)$ for the remaining values of $p,q\in [1,\infty]$ are as follows.
\begin{equation}\label{eq:q<p<2}
p,q\in [1,2]\ \mathrm{and}\ q\neq 1\implies \left(\frac{\log n}{\log \log n}\right)^{\frac{1}{p}-\frac12}\lesssim_q e_n(\ell_p,\ell_q)\lesssim_p (\log n)^{\frac{1}{p}},
\end{equation}
\begin{equation}\label{eq:q=1 p<2}
1\le p\le 2\implies \frac{(\log n)^{\frac{1}{p}-\frac12}}{(\log\log n)^{\frac{1}{p}}}\lesssim e_n(\ell_p,\ell_1)\lesssim_p(\log n)^{\frac{1}{p}}.
\end{equation}
\begin{equation}\label{eq:p,q>2}
p,q\in (2,\infty)\implies \left(\frac{\log n}{\log \log n}\right)^{\frac{q-2}{q^2}}\lesssim e_n(\ell_p,\ell_q)\lesssim \frac{\log n}{\log \log n},
\end{equation}
\begin{equation}\label{eq:p<2<q}
1\le p\le 2\le q<\infty \implies \left(\frac{\log n}{\log\log n}\right)^{\max\left\{\frac{1}{p}-\frac12,\frac{q-p}{q^2}\right\}}\lesssim_q e_n(\ell_p,\ell_q)\lesssim_p (\log n)^{\frac{1}{p}}.
\end{equation}

The right hand inequalities in~\eqref{eq:q<p<2}, \eqref{eq:q=1 p<2}, \eqref{eq:p<2<q} are due to~\cite{LN05}. More generally, it was shown in~\cite{LN05} that for $p\in (1,2]$ we have  $e_n(\ell_p,Z)\lesssim_p (\log n)^{1/p}$ for every Banach space $Z$. It would be interesting to determine the best exponent of $\log n$ in this context of general Banach space targets, as well as when the target $Z$ is allowed to range over some  special classes of Banach spaces (e.g., for Banach spaces of cotype $2$ it is known~\cite{MN13} that this exponent cannot be smaller than $1/4$).

No upper bound on $e_n(\ell_p,\ell_q)$ that is asymptotically smaller than the general upper bound~\cite{LN05} of $O((\log n)/\log\log n)$ is known when $p,q\in (2,\infty)$, and similarly for $e_n(\ell_1,\ell_q)$ for $q\in (2,\infty)$.

Let $(Z,\|\cdot\|_Z)$ be a Banach space that has Rademacher type $p$ for some $p>1$. By a theorem of Figiel and Tomczak-Jaegermann~\cite{FT79} there exists $K(Z)\in (0,\infty)$ such that for every $n\in \N$ one can find an $n$-dimensional subspace $Z_n$ of $Z$ that is $2$-isomorphic to $\ell_2^n$ and there exists a projection from $Z$ onto $Z_n$ of norm at most $K(Z)$. This implies that for any metric space $X$ we have $e_n(X,Z)\gtrsim_Z e_n(X,\ell_2)$. In particular, $e_n(Z,\ell_q)\gtrsim_q e_n(X,\ell_2)$ for every $q>1$. This implies the validity of the left hand inequality in~\eqref{eq:q<p<2}, as a consequence of the  left hand inequality of~\eqref{eq:MP}. However, this reasoning does not apply when the target space is $\ell_1$, i.e., in order to prove the left hand inequality in~\eqref{eq:q=1 p<2}. For this purpose, one argues by adapting the proof of the left hand inequality in~\eqref{eq:1 1} that appears in~\cite{MM10} (the key tool being~\cite{FJS88}); the adaptation of this argument is simple and we omit it (the result itself, however, is far from trivial).

If $(X,\|\cdot\|_X)$ is an infinite dimensional Banach space then, by Dvoretzky's theorem~\cite{Dvo60}, $X$ contains a $2$-isomorphic copy of $\ell_2^n$ for every $n\in \N$. Consequently, $e_n(X,Z)\gtrsim e_n(\ell_2,Z)$ for every metric space $Z$. In particular, the left hand inequality in~\eqref{eq:p,q>2} follows from the special case $p=2$ of the left hand inequality in~\eqref{eq:p<2<q}.  The latter inequality consists of two lower bounds, one with the exponent $1/p-1/2$ and the other with the exponent $(q-p)/q^2$. Since, as explained in the previous paragraph, $e_n(\ell_p,\ell_q) \gtrsim_q e_n(\ell_p,\ell_2)$, the lower bound with exponent  $1/p-1/2$  follows from the left hand inequality in~\eqref{eq:MP}. The lower bound with exponent $(q-p)/q^2$ is due to~\cite{Nao01}, with the explicit asymptotic dependence being computed when $p=2$ in~\cite[Lemma~1.13]{LN05}, and when $p\in (1,2)$ the corresponding bound follows mutatis mutandis by the same argument.

It is worthwhile to point out here an interesting feature of the left hand inequality in~\eqref{eq:p<2<q}. Suppose that $p\in [1,2]$. For every $n$-dimensional subspace $Y$ of $\ell_p$ there exists a projection from $\ell_p$ onto $Y$ of norm at most $c_p n^{1/p-1/2}$. This assertion follows from the work of Lewis~\cite{Lew78}, and it is also a consequence~\eqref{eq:KRT}. Thus, for every Banach space $Z$, every linear operator $U:Y\to Z$ can be extended to a linear operator $V:\ell_p\to Z$ with $\|V\|_{\ell_p\to Z}\lesssim_p n^{1/p-1/2}\|U\|_{Y\to Z}$. Consequently, the linearization method of Johnson and Lindenstrauss~\cite{JL84} cannot yield a lower bound on $e_n(\ell_p,Z)$ that is at least a constant multiple of $(\log n)^{1/p-1/2}$.  However, $(q-p)/q^2>1/p-1/2$ if and only if
\begin{equation}\label{eq:beat linearization}
\frac{3}{2}<p\le 2\qquad\mathrm{and}\qquad 2+\frac{\left(1-\sqrt{2p-3}\right)^2}{1+\sqrt{2p-3}}<q<2+\frac{\left(1-\sqrt{2p-3}\right)^3}{2(2-p)}.
\end{equation}
So, if $p,q$ satisfy~\eqref{eq:beat linearization} then the left hand inequality in~\eqref{eq:p<2<q} shows that
\begin{equation}\label{eq:better than lin}
\lim_{n\to \infty}\frac{e_n(\ell_p,\ell_q)}{(\log n)^{\frac{1}{p}-\frac12 }}=\infty,
\end{equation}
i.e., we have a lower bound on $e_n(\ell_p,\ell_q)$ that is asymptotically larger than any bound that can be deduced by reducing the problem to the extension problem for linear operators. We conjecture that the restrictions in~\eqref{eq:beat linearization} are not needed here, i.e., \eqref{eq:better than lin} holds true whenever $1\le p\le 2<q<\infty$.  It even seems to be unknown whether for every $p\in [1,2]$ there exists a Banach space $Z$ for which
\begin{equation}\label{eq:Z exists}
\lim_{n\to \infty}\frac{e_n(\ell_p,Z)}{(\log n)^{\frac{1}{p}-\frac12 }}=\infty.
\end{equation}
The best-known result towards~\eqref{eq:Z exists} seems to follow from~\cite{MN13}, where it is shown that  there exists a Banach space $Z$ for which $e_n(\ell_2,Z)\gtrsim \sqrt[4]{(\log n)/\log\log n}$. Hence, by Dvoretzky's theorem~\cite{Dvo60}, for every infinite dimensional Banach space $X$ we have $e_n(X,Z) \gtrsim \sqrt[4]{(\log n)/\log\log n}$, implying that~\eqref{eq:Z exists} holds true if $p\in [1,2]$ satisfies $1/p-1/2<1/4$, i.e., for every $p\in (4/3,2]$.

\section{Algorithmic clustering}\label{sec:algorithmic} An algorithmic optimization problem called {\em $0$-Extension}, which we describe below, served as inspiration for our proof of Theorem~\ref{thm:main}. In this section we shall survey this background so as to clarify the context and explain the ``twist" over the existing approaches to $0$-Extension that we introduce in order to prove Theorem~\ref{thm:main}. We stress, however, that our present work does not have new algorithmic implications, and the sole purpose of this section is to explain how this context motivated our approach to Theorem~\ref{thm:main}. Those who are interested only in the proof of Theorem~\ref{thm:main} can skip this section on first reading: the proof itself appears in Section~\ref{sec:proof main} below and is entirely self-contained, with none of the algorithmic background that we describe here being used.

In what follows, the vertices of a combinatorial graph $G$ are denoted $V_G$ and its edges are denoted $E_G$. The $0$-Extension problem is a clustering framework for finite graphs that was introduced by Karzanov~\cite{Kar98}. The input of the $0$-Extension problem is a graph $G$ with edge weights $w:E_G\to [0,\infty)$, and a subset $T\subset V_G$ equipped with a metric $d_T:T\times T\to [0,\infty)$. The subset $T$ is called in the literature the set of {\em terminals}. The output of the $0$-Extension problem is a partition of $V_G$ into $|T|$ subsets $\{C_x\subset V_G\}_{x\in T}$ with the requirement that $x\in C_x$ for every $x\in T$. The cost of this partition is defined as follows. Every edge $\{u,v\}\in E_G$ with $|\{u,v\}\cap C_x|=|\{u,v\}\cap C_y|=1$ for some distinct $x,y\in T$ contributes its weight $w(u,v)$ times $d_T(x,y)$ to the total cost, and all other edges (i.e., edges that are entirely within one cluster $C_x$ for some $x\in T$) do not contribute to the total cost. The goal is to find efficiently (in polynomial time) such a partition with minimum cost, or with cost that is guaranteed to be close to the minimum possible cost. Formally, define
$$
\mathrm{COST}_{(G,w,d_T)}\left(\{C_x\}_{x\in T}\right)\eqdef \sum_{\substack{x,y\in T\\ x\neq y}}\sum_{\substack{\{u,v\}\in E_G\\ |\{u,v\}\cap C_x|=|\{u,v\}\cap C_y|=1}} w(u,v) d_T(x,y),
$$
and
\begin{equation}\label{eq:def opt}
\mathrm{OPT}(G,w,d_T)\eqdef \min_{\substack{\{C_x\}_{x\in T}\ \mathrm{partition\ of\ } V_G\\ \forall\, x\in T,\quad x\in C_x}}\mathrm{COST}_{(G,w,d_T)}\left(\{C_x\}_{x\in T}\right).
\end{equation}
The goal is to find efficiently a partition $\{C_x\}_{x\in T}$ of $V_G$ with $x\in C_x$ for every $x\in T$ such that
$$
\frac{\mathrm{COST}_{(G,w,d_T)}\left(\{C_x\}_{x\in T}\right)}{\mathrm{OPT}(G,w,d_T)}
$$
is guaranteed to be at most a (hopefully small) value $\alpha\in [1,\infty)$.

Obtaining $\alpha=1$ here would imply that $P=NP$, since in the special when $w(u,v)= 1$ and $d_T(x,y)=1$ for every $\{u,v\}\in E_G$ and every distinct $x,y\in T$, the quantity $\mathrm{COST}_{(G,w,d_T)}\left(\{C_x\}_{x\in T}\right)$ is nothing more than the total number of edges that are incident to distinct elements of the partition $\{C_x\}_{x\in T}$. Computing the minimum cost in this special case is the  MULTIWAY CUT problem, and it was shown in~\cite{DJPSY94} that if $P\neq NP$ then there exists $\alpha_0>1$ such that no polynomial-time algorithm outputs a partition whose cost is guaranteed to be less than $\alpha_0$ times the minimum possible cost. For the $0$-Extension problem in its full generality, it was shown in~\cite{KKMR09} that for every $\e\in (0,1)$ there exists $C(\e)\in (0,\infty)$ such that the existence of a polynomial-time  algorithm for $0$-Extension whose approximation factor on instances of size $n$   is $\alpha\le (\log n)^{1/4 -\e}$ would imply that any problem in $NP$ of size $n$ could be solved in time $\exp((\log n)^{C(\e)})$.

Due to the above  evidence for the nonexistence of a constant-factor approximation algorithm for  $0$-Extension, the literature has focused on the design of approximation algorithms with potentially unbounded approximation factor, based on two competing continuous relaxations of the discrete optimization problem. The first is the {\em metric relaxation}, which was formulated in~\cite{Kar98} and studied extensively in~\cite{CKR04}. Given an instance of $0$-Extension, i.e., an $n$-vertex graph $G$, a metric $d_T:T\times T\to [0,\infty)$ on a subset $T\subset V_G$, and edge weights $w:E_G\to [0,\infty)$, define
\begin{equation}\label{eq:def metric relaxation}
\mathrm{MET}(G,w,d_T)\eqdef \min_{\substack{d:V_G\times V_G\to [0,\infty)\\ d|_T=d_T\\ d \mathrm{\ is\ a\ semi-metric}}}\sum_{\{u,v\}\in E_G} w(u,v)d(u,v).
\end{equation}
The minimization in the right hand side of~\eqref{eq:def metric relaxation} amounts to a linear program because the constraint that $d:V_G\times V_G\to [0,\infty)$ is a semi-metric that extends $d_T$ corresponds to $O(n^3)$ linear inequalities in the variables $\{d(u,v)\}_{u,v\in V_G}$. Therefore $\mathrm{MET}(G,w,d_T)$ can be computed in polynomial time.

A second relaxation for the $0$-Extension problem, called the {\em earthmover relaxation}, was proposed in~\cite{CKNZ04} and studied extensively in~\cite{AFHKTT04,KKMR09}. The idea is similar to~\eqref{eq:def metric relaxation}, except that the semi-metric $d:V_G\times V_G\to [0,\infty)$  is further restricted to a special class of semi-metrics, over which the corresponding minimization can still be cast as a linear program. Given two measures $\mu,\nu$ on $T$ with the same total mass, i.e., $\mu(T)=\nu(T)$, let $\Pi(\mu,\nu)$ be the set of all couplings of $\mu$ and $\nu$. Thus, $\Pi(\mu,\nu)$ consists of all the measures $\pi$ on $T\times T$ such that $\sum_{y\in T} \pi(x,y)=\mu(x)$ and $\sum_{y\in T}\pi(y,x)=\nu(x)$ for every $x\in T$. The assumption that $\mu$ and $\nu$ have the same total mass ensures that $\Pi(\mu,\nu)\neq\emptyset$; specifically $(\mu\times \nu)/\mu(T)\in \Pi(\mu,\nu)$. The Wasserstein $1$ distance between $\mu$ and $\nu$ (also known as the earthmover distance between $\mu$ and $\nu$) is defined to be
\begin{equation}\label{eq:def W1 T}
W_1^{d_T}(\mu,\nu)\eqdef \min_{\pi\in \Pi(\mu,\nu)} \sum_{x,y\in T} d_T(x,y)\pi(x,y).
\end{equation}
Let $\mathscr{P}_T$ denote the set of all probability measures on $T$ and define
\begin{equation}\label{eq:def EMD relaxation}
\mathrm{EMD}(G,w,d_T)\eqdef \min_{\substack{\{\mu_u\}_{u\in V_G}\subset \mathscr{P}_T\\\forall\, x\in T,\quad  \mu_x=\delta_x}}\sum_{\{u,v\}\in E_G} w(u,v)W_1^{d_T}(\mu_u,\mu_v).
\end{equation}
Here for $x\in T$ the point mass at $x$ is denoted $\delta_x$. The minimization in the right hand side of~\eqref{eq:def EMD relaxation} also amounts to a linear program, with variables corresponding to couplings $\pi_{uv}$ of the probability measures $\mu_u$ and $\mu_v$ for each $\{u,v\}\in E_G$, and objective $\sum_{\{u,v\}\in E_G}\sum_{x,y\in T} w(u,v)d_T(x,y)\pi_{uv}(x,y)$; the requirements of being a probability measure or a coupling are clearly linear constraints. The quantity $\mathrm{EMD}(G,w,d_T)$ can therefore be computed in polynomial time.

If $\{\mu_u\}_{u\in V_G}\subset \mathscr{P}_T$ are such that $\mu_x=\delta_x$ for every $x\in T$ then $d(u,v)=W_1^{d_T}(\mu_u,\mu_v)$ is a semi-metric on $V_G$ that extends $d_T$. Consequently, the minimization in~\eqref{eq:def EMD relaxation} is over a smaller set than the minimization in~\eqref{eq:def metric relaxation}. At the same time, if $\{C_x\}_{x\in T}$ is a partition of $V_G$ with $x\in C_x$ for every $x\in T$, then by defining $\mu_u=\delta_x$ for every $(u,x)\in V_G\times T$ we see that the minimization in~\eqref{eq:def opt} is over a smaller set than the minimization in~\eqref{eq:def EMD relaxation}. These observations show that every instance of $0$-Extension satisfies
\begin{equation}\label{eq:relaxations}
\mathrm{MET}(G,w,d_T)\le \mathrm{EMD}(G,w,d_T)\le \mathrm{OPT}(G,w,d_T).
\end{equation}

Building on ideas of~\cite{CKR04}, it was shown in~\cite{FHRT03}  that every instance of $0$-Extension satisfies $\mathrm{OPT}(G,w,d_T)/\mathrm{MET}(G,w,d_T)\lesssim (\log |T|)/\log\log |T|$. This yields the best known approximation algorithm for the $0$-Extension problem, the algorithm being to compute the quantity $\mathrm{MET}(G,w,d_T)$, i.e., to apply the metric relaxation. It follows from~\eqref{eq:relaxations} that by using the earthmover relaxation, i.e., by computing $\mathrm{EMD}(G,w,d_T)$, one could potentially obtain an even better algorithm. This was realized in~\cite{AFHKTT04} in special cases (e.g.~when $(T,d_T)$ is a planar graph), but in~\cite{LN03} it was shown that the same approximation guarantees as in~\cite{AFHKTT04} can be obtained by using the metric relaxation. Nevertheless, in~\cite{KKMR09} it was shown that in a certain sense  the earthmover relaxation of $0$-Extension does perform better than the metric relaxation: the earthmover relaxation behaves better than the metric relaxation if one measures the approximation factor in term of the ratio between the largest distance and the smallest nonzero distance in $(T,d_T)$.

On the negative side, it was shown in~\cite{CKR04} that there exist instances of $0$-Extension with $|T|$ arbitrarily large and $\mathrm{OPT}(G,w,d_T)/\mathrm{MET}(G,w,d_T)\gtrsim \sqrt{\log |T|}$. More recently, in~\cite{KKMR09} it was shown that there also exist such instances with $\mathrm{OPT}(G,w,d_T)/\mathrm{EMD}(G,w,d_T)\gtrsim \sqrt{\log |T|}$.

As described above, the earthmover relaxation of $0$-Extension is in principle better than the metric relaxation, but at present the general bounds that are available in the literature are  to a large extent the same for both relaxations. Nevertheless, our approach to Theorem~\ref{thm:main} is based on exploiting the {\em differences} between the metric relaxation and the earthmover relaxation. The idea is that if an instance of $0$-Extension behaves significantly worse for the metric relaxation than for the earthmover relaxation then the terminal metric $(T,d_T)$ can be embedded isometrically into a larger semi-metric $(V_G,d)$ for which $\sum_{\{u,v\}\in E_G} w(u,v)d(u,v)$ is much smaller than $\mathrm{OPT}(G,w,d_T)$, yet whenever  one assigns to every vertex $u\in V_G$ a probability measure $\mu_u\in \mathscr{P}_T$ such that $\mu_x=\delta_x$ for every $x\in T$, then $\sum_{\{u,v\}\in E_G} w(u,v)W_1^{d_T}(\mu_u,\mu_v)$ must be close to $\mathrm{OPT}(G,w,d_T)$. This discrepancy between the two relaxations can be used to show that the mapping that assigns to every $x\in T$ the point mass $\delta_x\in \mathscr{P}_T$ cannot be extended to a function from $(V_G,d)$ to $(\mathscr{P}_T,W_1^{d_T})$ that has a small Lipschitz constant. Here the target space $(\mathscr{P}_T,W_1^{d_T})$  is not a normed space, but this idea can be modified so as to yield a poorly extendable function with values in a related normed space, namely the dual of the space of real-valued Lipschitz functions on $(T,d_T)$ that vanish at a fixed point. In the end, the entire argument as presented in Section~\ref{sec:proof main} uses geometric and combinatorial considerations that are self-contained and do not make any reference to the clustering objective $\mathrm{OPT}(G,w,d_T)$. The examples of metric spaces that we use to prove Theorem~\ref{thm:main} are modifications of the examples that were considered in~\cite{CKR04}, which are themselves in the spirit of an example that was used in~\cite{JLS86} (and, we use yet another variant of the example of~\cite{JLS86} to prove Theorem~\ref{thm:holder}).

\section{Proof of Theorem~\ref{thm:main}}\label{sec:proof main}

Before passing to the proof of Theorem~\ref{thm:main}, which appears in Section~\ref{sec:the proof here} below, we need to recall some (simple) background and to introduce some basic constructions.

\subsection{The $r$-magnification of a metric space}\label{sec:magnification} Given a metric space $(X,d_X)$ and $r>0$, for every subset $S\subset X$ we shall define a new metric space $X_r(S)$ that we call the $r$-magnification of $(X,d_X)$ at $S$. As a set, $X_r(S)$ equals $X$. The new metric $d_{X_r(S)}$ on $X$ is defined by setting  $d_{X_r(S)}(x,x)=0$ for every $x\in X$, and by defining for every distinct $x,y\in X$,
\begin{equation}\label{eq:def XrS}
d_{X_r(S)}(x,y)\eqdef d_X(x,y)+r|\{x,y\}\cap S|.
\end{equation}
It is immediate to check that $d_{X_r(S)}$ is indeed a metric on $S$. Note that this construction increases all the positive pairwise distances within $S$ while keeping the pairwise distances within $X\setminus S$ unchanged. This is the reason for our choice of terminology. If $G$ is a connected graph then $d_G$ will always stand for the shortest-path metric on $V_G$, and for $S\subset V_G$ and $r>0$, the metric space $G_r(S)$ will always be understood to be the $r$-magnification of $(V_G,d_G)$ at $S$.

The relevance of the $r$-magnification of $(X,d_X)$ at $S\subset X$ to potentially proving impossibility results for Lipschitz extension is simple to explain. By making the positive pairwise distances within $S$ larger, we make it easier for functions that are defined on $S$ to be Lipschitz, so there are more ``potential counterexamples" on $(S,d_{X_r(S)})$ than there were on $(S,d_X)$. At the same time, by keeping the pairwise distances within $X\setminus S$ unchanged, we do not make the Lipschitz condition on $V\setminus S$ any less stringent. However, there is a limitation to this reasoning because as $r$ becomes larger the minimum positive distance in $(S,d_{X_r(S)})$ becomes closer to the diameter of $(S,d_{X_r(S)})$, in which case bounds on $\Ae(\e)$ become relevant. Indeed, by \eqref{eq:def XrS} any distinct $x,y\in S$ satisfy
\begin{equation}\label{eq:aspect in magnification}
d_{X_r(S)}(x,y)\ge 2r+1= \frac{2r+1}{2r+\diam(S,d_X)}\diam(S,d_{X_r(S)}).
\end{equation}
Since $\Ae(\e)\lesssim 1/\e$, it follows that any $1$-Lipschitz function from $(S,d_{X_r(S)})$ to any Banach space can be extended to a function defined on $(X,d_{X_r(S)})$ whose Lipschitz constant is at most a constant multiple of $(2r+ \diam(S,d_X))/(2r+1)\le 1+\diam(S,d_X)/r$. There is therefore a tradeoff that limits how large $r$ could be if one wishes to use the $r$-magnification  for the purpose of obtaining a lower bound on $\Ae(n)$. Below we shall balance these constraints, for an appropriate choice of an initial metric space $(X,d_X)$, so as to yield Theorem~\ref{thm:main}.

\subsection{The Wasserstein $1$ norm} For a finite set $X$ we denote (as usual) by $\R^X$ the $|X|$-dimensional vector space of all $f:X\to \R$. We also denote by $\R^X_0$ the subspace of $\R^X$ consisting of those $f:X\to \R$ that satisfy $\sum_{x\in X} f(x)=0$. The standard basis of $\R^X$ will be denoted by $\{e_x\}_{x\in X}$, i.e., $e_x(y)=\1_{\{x=y\}}$ for every $x,y\in X$. For $f\in \R^X$ we shall use the (standard) notation $$\|f\|_{\ell_1(X)}\eqdef \sum_{x\in X}|f(x)|.$$

If $(X,d_X)$ is a finite metric space then let $K_{(X,d_X)}\subset \R^X_0$ be the the following convex hull.
\begin{equation}\label{def K unit ball}
K_{(X,d_X)}\eqdef\conv\left\{\frac{e_x-e_y}{d_X(x,y)}:\ x,y\in X\ \mathrm{and}\   x\neq y\right\}.
\end{equation}
$K_{(X,d_X)}$ is clearly an origin-symmetric convex body, so it is a unit ball of a norm on $\R^X_0$, called the Wasserstein $1$ norm induced by $X$, which we denote by $\|\cdot\|_{W_1(X,d_X)}$. By the Kantorovich--Rubinstein duality theorem (see~\cite[Thm.~1.14]{Vil03}), denoting as usual $f^+= \max\{f,0\}$ and $f^-= \max\{-f,0\}$ for every $f\in \R^X$, and recalling~\eqref{eq:def W1 T}, we have
\begin{equation}\label{eq:w1 norm kantorovich}
\forall\, f\in \R^X_0,\qquad \|f\|_{W_1(X,d_X)}=W_1^{d_X}(f^+,f^-)=\inf_{\pi\in \Pi(f^+,f^-)}\sum_{x,y\in X} d_X(x,y)\pi(x,y).
\end{equation}
Observe that this makes sense because the assumption $f\in \R^X_0$, i.e., that $\sum_{x\in X} f(x)=0$, implies that the nonnegative functions $f^+$ and $f^-$ satisfy, $\sum_{x\in X}f^+(x)=\sum_{x\in X}f^-(x)$. So, $f^+$ and $f^-$ are nonnegative measures with the same total mass.

We record for future use the following very simple lemma.
\begin{lemma}\label{lem:w1l1}
Let $(X,d_X)$ be a finite metric space. Then $\|e_x-e_y\|_{W_1(X,d_X)}=d_X(x,y)$ for every $x,y\in X$ and
\begin{equation}\label{eq:minimal distanmce w1}
\forall\, f\in \R_0^X,\qquad \frac12\min_{\substack{x,y\in X\\x\neq y}} d_X(x,y)\|f\|_{\ell_1(X)}\le \|f\|_{W_1(X,d_X)}\le \frac12\diam(X)\|f\|_{\ell_1(X)}.
\end{equation}
In particular, for every $r>0$ and $S\subset X$ we have
\begin{equation}\label{eq:wasserstein lower on magnified}
\forall\, f\in \R_0^S,\qquad r\|f\|_{\ell_1(S)}\le \|f\|_{W_1(S,d_{X_r(S)})}\le \left(r+\frac{\diam(X)}{2}\right)\|f\|_{\ell_1(S)} .
\end{equation}
\end{lemma}

\begin{proof}
The fact that $\|e_x-e_y\|_{W_1(X,d_X)}=d_X(x,y)$ for every distinct $x,y\in X$ is immediate from~\eqref{eq:w1 norm kantorovich}, and also directly from~\eqref{def K unit ball}, which says that $(e_x-e_y)/d_X(x,y)$ has unit norm in $W_1(X,d_X)$. To prove~\eqref{eq:minimal distanmce w1}, denote by $m$ the minimum nonzero value that $d_X$ attains. For distinct $x,y\in X$,
$$
\max_{\substack{x,y\in X\\x\neq y}} \left\|\frac{e_x-e_y}{d_X(x,y)}\right\|_{\ell_1(X)}\le \max_{\substack{x,y\in X\\x\neq y}} \frac{\|e_x-e_y\|_{\ell_1(X)}}{m}=\frac{2}{m}.
$$
This means that for every distinct $x,y\in X$ we have $(e_x-e_y)/d_X(x,y)\in (2/m)B_{\ell_1(X)}$, where $B_{\ell_1(X)}$ is the unit ball of $\ell_1(X)$. By~\eqref{def K unit ball} we therefore have $K_{(X,d_X)}\subset (2/m)B_{\ell_1(X)}$, which is the same as the first inequality in~\eqref{eq:minimal distanmce w1}. The second inequality in~\eqref{eq:minimal distanmce w1} follows directly from~\eqref{eq:w1 norm kantorovich}, noting that the sum in the right hand side of~\eqref{eq:w1 norm kantorovich} is at most $\diam(X)\sum_{x\in X} f^+(x)=\diam(X)\|f\|_{\ell_1(X)}/2$. The estimate~\eqref{eq:wasserstein lower on magnified} is a special case of~\eqref{eq:minimal distanmce w1} because by~\eqref{eq:def XrS} the minimum nonzero distance within $S$ of the $r$-magnification of $X$ at $S$ is at least $2r$, and the diameter of $X_r(S)$ is at most $2r+\diam(X)$.
\end{proof}

\subsection{Properties of expanders} In the proof of Theorem~\ref{thm:main} we shall use several properties of graphs in general, and expander graphs in particular. Here we collect these facts for ease of later reference. Fix two integers $n,d\ge 3$ and let $G$ be a connected $n$-vertex $d$-regular graph. Recall that the shortest-path metric that $G$ induces on $V_G$ is denoted $d_G$. The following estimate is standard.
\begin{equation}\label{eq:average distance}
\forall\, \emptyset \neq S\subset V_G,\qquad \frac{1}{|S|^2}\sum_{x,y\in S} d_G(x,y)\ge \frac{\log |S|}{4\log d}.
\end{equation}
To quickly verify the validity of~\eqref{eq:average distance}, observe that since the smallest nonzero distance in $G$ is at least $1$, the average that appears in the left hand side of~\eqref{eq:average distance} is always at least $|S|(|S|-1)/|S|^2=1-1/|S|$. One checks directly that $1-1/a \ge (\log a)/(4\log 3)$ for every $a\in \{1,\ldots,15\}$, so in order to establish~\eqref{eq:average distance} we may assume that $|S|\ge 16$. Denote $k=1+\lfloor \log _d (|S|/4)\rfloor$ and observe that for every $x\in V_G$ the number of $y\in V_G$ with $d_G(x,y)\le k-1$ is at most $1+d+\ldots+d^{k-1}\le 2d^{k-1}\le |S|/2$. The average that appears in the left hand side of~\eqref{eq:average distance} is therefore at least $k/2$, and it remains to note that $k\ge \log _d (|S|/4)\ge (\log|S|)/(2\log d)$, since $|S|\ge 16$.

Fixing $S\subset V_G$ and $r>0$, we shall need later the following straightforward evaluation of the average length of edges of $G$ in the $r$-magnification of $(G,d_G)$ at $S$.
\begin{equation}\label{eq:magnifies edge sum}
\frac{1}{|E_G|}\sum_{\{x,y\}\in E_G}d_{G_r(S)}(x,y)= 1+\frac{2r|S|}{n}.
\end{equation}
Indeed, let $E_1$ be those edges in $E_G$ that are contained in $S$ and let $E_2$ be those edges in $E_G$ that contain exactly one element of $S$. Because $G$ is $d$-regular, we have $2|E_1|+|E_2|=d|S|$.  Recalling~\eqref{eq:def XrS}, for every $\{x,y\}\in E_1$ we have $d_{G_r(S)}(x,y)= 2r+1$, for every $\{x,y\}\in E_2$ we have $d_{G_r(S)}(x,y)= r+1$, and for every  $\{x,y\}\in E_G\setminus (E_1\cup E_2)$ we have $d_{G_r(S)}(x,y)=1$. Consequently,
\begin{align*}
\frac{1}{|E_G|}\sum_{(x,y)\in E_G}d_{G_r(S)}(x,y)&= \frac{(|E_G|-|E_1|-|E_2|)+(2r+1)|E_1|+(r+1)|E_2|}{|E_G|}\\&=1+\frac{r(2|E_1|+|E_2|)}{dn/2}=1+\frac{2r|S|}{n}.
\end{align*}

Given two disjoint subsets $S,T\subset V_G$, denote the number of edges in $E_G$ that intersect both $S$ and $T$ by $E_G(S,T)$. The edge-expansion of $G$, denoted $\phi(G)$, is the largest $\phi\in [0,\infty)$ such that
\begin{equation}\label{eq:def expansion}
\forall\, S\subset V_G,\qquad E_G(S,V_G\setminus S)\ge \phi \frac{|S|(n-|S|)}{n^2}|E_G|.
\end{equation}
It is well known that~\eqref{eq:def expansion} is equivalent to the assertion that every $h:V_G\to \ell_1$ satisfies.
\begin{equation}\label{eq:cheeger expansion}
\frac{\phi}{n^2}\sum_{x,y\in V_G}\|h(x)-h(y)\|_{1}\le \frac{1}{|E_G|}\sum_{\{x,y\}\in E_G}\|h(x)-h(y)\|_1.
\end{equation}
The equivalence of~\eqref{eq:def expansion} and~\eqref{eq:cheeger expansion} is a standard application of the cut-cone decomposition of subsets of $\ell_1$; see e.g.~inequality (4) in~\cite{Mat97} or~\cite[Fact~2.1]{NRS05}. The following simple combination of~\eqref{eq:cheeger expansion} and Lemma~\ref{lem:w1l1} will be used later.

\begin{lemma}\label{lem:poincare into wasserstein}
Fix $n\in \N$ and $\phi\in (0,1]$. Suppose that $G$ be an $n$-vertex graph with $\phi(G)\ge \phi$. For every $\emptyset \neq S\subset V_G$ and $r>0$, every $F:V_G\to \R_0^S$ satisfies
$$
\frac{1}{n^2}\sum_{x,y\in V_G}\left\|F(x)-F(y)\right\|_{W_1(S,d_{G_r(S)})}\le \frac{2r+\diam(S,d_G)}{(2r+1)\phi}\cdot \frac{1}{|E_G|}\sum_{\{x,y\}\in E_G}\left\|F(x)-F(y)\right\|_{W_1(S,d_{G_r(S)})}.
$$
\end{lemma}

\begin{proof}
By~\eqref{eq:cheeger expansion} we have
\begin{equation}\label{eq:cheeger in ell1S}
\frac{1}{n^2}\sum_{x,y\in V_G}\left\|F(x)-F(y)\right\|_{\ell_1(S)}\le  \frac{1}{\phi|E_G|}\sum_{\{x,y\}\in E_G}\left\|F(x)-F(y)\right\|_{\ell_1(S)}
\end{equation}
Recalling~\eqref{eq:def XrS}, we have $\diam(S,d_{G_r(S)})=2r+\diam(S,d_G)$ and the smallest positive distance in $(S,d_{G_r(S)})$ equals $2r+1$. So, by Lemma~\ref{lem:w1l1}, every $x,y\in V_G$ satisfy
\begin{equation*}
\frac{2r+1}{2}\left\|F(x)-F(y)\right\|_{\ell_1(S)}\le \left\|F(x)-F(y)\right\|_{W_1(S,d_{G_r(S)})}\le \frac{2r+\diam(S,d_G)}{2}\left\|F(x)-F(y)\right\|_{\ell_1(S)}.
\end{equation*}
Lemma~\ref{lem:poincare into wasserstein} now follows by substituting these estimates for pairwise distances into~\eqref{eq:cheeger in ell1S}.
\end{proof}

We end by recording for future use the following direct consequence of Menger's theorem~\cite{Men27}.

\begin{lemma}\label{lem:menger}
Let $G$ be an $n$-vertex graph and $A,B\subset V_G$ satisfy $A\cap B=\emptyset$. Fix $\phi\in (0,\infty)$ and suppose that $\phi(G)\ge \phi$. Then the number of edge-disjoint paths joining $A$ and $B$ is at least $\phi \min\{|A|,|B|\}|E_G|/(2n)$. In other words, there exists an integer $m\ge \phi \min\{|A|,|B|\}|E_G|/(2n)$, $k_1,\ldots,k_m\in \N$ and $\{u_{i,1},u_{i,2},\ldots,u_{i,k_i}\}_{i=1}^m\subset V_G$ such that $\{u_{i,1}\}_{i=1}^m\subset A$, $\{u_{i,k_i}\}_{i=1}^m\subset B$.  Moreover, $\{u_{i,s},u_{i,s+1}\}\in E_G$ for every $i\in \{1,\ldots,m\}$ and $s\in \{1,\ldots,k_i-1\}$, and if for some $i,j\in \{1,\ldots,m\}$, $s\in \{1,\ldots,k_i-1\}$ and $t\in \{1,\ldots,k_j-1\}$ we have $\{u_{i,s},u_{i,s+1}\}= \{u_{j,t},u_{j,t+1}\}$ then necessarily $i=j$ and $s=t$.
\end{lemma}

\begin{proof}
Let $m$ be the maximal number of edge-disjoint paths joining $A$ and $B$. By the classical Menger theorem~\cite{Men27} (see also e.g.~\cite[Chapter~3]{Die10}) there exists a subset of edges $E^*\subset E_G$ with $|E^*|=m$ such that every path in $G$ that joins a vertex in $A$ with a vertex in $B$ contains an edge from $E^*$.  Since in the graph $G^*=(V_G,E_G\setminus E^*)$ there is no path that joins an element of $A$ with an element of $B$, if we let $C\subset V_G$ be the union of all the connected components of $G^*$ that contain an element of $A$ then $C\supset A$ and $C\cap B=\emptyset$. Since $C$ is a union of connected components of $G^*$, all the edges of $E_G$ joining $C$ and $V_G\setminus C$ belong to $E^*$. Hence, $E_G(C, V_G\setminus C)\le |E^*|=m$. It remains to note that by~\eqref{eq:def expansion} we have
\begin{equation}\label{eq:use expansion def}
m\ge E_G(C,V_G\setminus C)\ge \phi \frac{\max\{|C|,n-|C|\}\cdot\min\{|C|,n-|C|\}}{n^2}|E_G|\ge \frac{\phi \min\{|A|,|B|\}|E_G|}{2n},
\end{equation}
where the last step of~\eqref{eq:use expansion def} holds true because clearly $\max\{|C|,n-|C|\}\ge n/2$, and since $C\supset A$ and $V_G\setminus C\supset B$ we have $\min\{|C|,n-|C|\}\ge \min\{|A|,|B|\}$.
\end{proof}

\subsection{A Wasserstein-valued poorly extendable function}\label{sec:the proof here} Fix $d,n\in \N$ and $\phi\in (0,1)$. Throughout this section, $G$ will be fixed to be an $n$-vertex $d$-regular graph with $\phi(G)\ge \phi$. We shall also fix $\emptyset \neq S\subset V_G$ and $r>0$. Define a mapping
$$
f:\left(S,d_{G_r(S)}\right)\to \left(\R^S_0,\|\cdot\|_{W_1(S,d_{G_r(S)})}\right)
$$
by
\begin{equation}\label{eq:def f wass}
\forall\, x\in S,\qquad f(x)\eqdef e_x-\frac{1}{|S|}\sum_{z\in S}e_z.
\end{equation}
Thus  $f$ is an isometry (recall Lemma~\ref{lem:w1l1}). Suppose that $F:V_G\to \R_0^S$ extends $f$ and  for some $L\in (0,\infty)$ we have
\begin{equation}\label{eq:F lipschitz}
\forall\, x,y\in V_G,\qquad \|F(x)-F(y)\|_{W_1(S,d_{G_r(S)})}\le L d_{G_r(S)}(x,y).
\end{equation}
Our goal is to bound $L$ from below.

For every $x\in V_G$ and $s\in (0,\infty)$ define $\B_s(x)\subset V_G$ to be the inverse image under $F$ of the $W_1(S,d_{G_r(S)})$-ball of radius $s$ centered at $F(x)$, i.e.,
$$
\B_s(x)\eqdef\left\{y\in V_G:\ \|F(x)-F(y)\|_{W_1(S,d_{G_r(S)})}\le s\right\}.
$$
By Lemma~\ref{lem:menger} there exists an integer
\begin{equation}\label{eq:m lower}
m\ge \frac{\phi d}{4}\min\left\{\left|S\setminus \B_s(x)\right|,\left|\B_s(x)\right|\right\},
\end{equation}
and $m$ edge-disjoint paths joining $S\setminus \B_s(x)$ and $\B_s(x)$. This means that we can find $k_1,\ldots,k_m\in \N$ and $\{z_{j,1},z_{j,2},\ldots,z_{j,k_j}\}_{j=1}^m\subset V_G$ such that $\{z_{j,1}\}_{j=1}^m\subset S\setminus \B_s(x)$, $\{z_{j,k_j}\}_{j=1}^m\subset \B_s(x)$, and such that  $\{\{z_{j,i},z_{j,i+1}\}: j\in \{1,\ldots,m\}\ \wedge\   i\in \{1,\ldots,k_j-1\}\}$ are distinct edges in $E_G$.

Let $J\subset \{1,\ldots,m\}$ be such that $\{z_{j,1}\}_{j\in J}$ are distinct and $\{z_{j,1}\}_{j\in J}=\{z_{i,1}\}_{i=1}^m$. For every $j\in J$ denote the number of those $i\in \{1,\ldots,m\}$ for which $z_{j,1}=z_{i,1}$ by $d_j$. Since $\{\{z_{i,1},z_{i,2}\}\}_{i=1}^m$ are distinct edges in $E_G$, and $G$ is $d$-regular,  $\max_{j\in J} d_j\le d$. Because $\sum_{j\in J} d_j=m$, it follows that
\begin{equation}\label{eq:J lower}
|J|\ge \frac{m}{d}\stackrel{\eqref{eq:m lower}}{\ge} \frac{\phi }{4}\min\left\{\left|S\setminus \B_s(x)\right|,\left|\B_s(x)\right|\right\}.
\end{equation}
The following lemma provides an upper bound on $|J|$ that we will later contrast with~\eqref{eq:J lower}.
\begin{lemma}
Under the above notation and assumptions we have
\begin{equation}\label{eq:J upper}
|J|\le \max\left\{d^{16(s-r)},\frac{16Lnd\log d}{\log n}\left(1+\frac{2r|S|}{n}\right)\right\}.
\end{equation}
\end{lemma}
\begin{proof} If $|J|\le d^{16(s-r)}$ then we are done, so we may assume that $|J|> d^{16(s-r)}$, or equivalently
\begin{equation}\label{eq:upper assumption on J}
s-r<\frac{\log |J|}{16\log d}.
\end{equation}

Observe that because $\{z_{j,1}\}_{j\in J}\subset S$ and $F|_S=f$ is an isometry on $(S,d_{G_r(S)})$,
\begin{equation}\label{eq:use isometry}
\forall\, i,j\in J,\qquad \|F(z_{i,1})-F(z_{j,1})\|_{W_1(S,d_{G_r(S)})}=d_{G_r(S)}(z_{i,1},z_{j,1})\stackrel{\eqref{eq:def XrS}}{=}2r+d_G(z_{i,1},z_{j,1}).
\end{equation}

Hence,
\begin{align}
\label{eq:sum in image} &\sum_{j\in J} \left\|F(z_{j,1})-F(z_{j,k_j})\right\|_{W_1(S,d_{G_r(S)})}\\\nonumber &=\frac{1}{2|J|} \sum_{i,j\in J}\left(\left\|F(z_{i,1})-F(z_{i,k_i})\right\|_{W_1(S,d_{G_r(S)})}+\left\|F(z_{j,1})-F(z_{j,k_j})\right\|_{W_1(S,d_{G_r(S)})}\right)\\\nonumber
&\ge \frac{1}{2|J|}\sum_{i,j\in J}\left(\left\|F(z_{i,1})-F(z_{j,1})\right\|_{W_1(S,d_{G_r(S)})}-\left\|F(z_{i,k_i})-F(z_{j,k_j})\right\|_{W_1(S,d_{G_r(S)})}\right)\\ \label{eq:use s ball}
& \ge \frac{1}{2|J|}\sum_{i,j\in J} d_G(z_{i,1},z_{j,1})-(s-r)|J|\\
&\ge \frac{|J|\log |J|}{8\log d}-(s-r)|J|>\frac{|J|\log |J|}{16\log d}, \label{eq:use average on J}
\end{align}
where in~\eqref{eq:use s ball} we used~\eqref{eq:use isometry} and the fact that, because $\{z_{j,k_j}\}_{j\in J}\subset \B_s(x)$, it follows from the definition of $\B_s(x)$ that for every $i,j\in J$ we have
\begin{equation*}
\|F(z_{i,k_i})-F(z_{j,k_j})\|_{W_1(S,d_{G_r(S)})}\le \|F(z_{i,k_i})-F(x)\|_{W_1(S,d_{G_r(S)})}+\|F(x)-F(z_{j,k_j})\|_{W_1(S,d_{G_r(S)})}\le 2s.
\end{equation*}
The penultimate inequality in~\eqref{eq:use average on J} is an application of~\eqref{eq:average distance}, and the final inequality in~\eqref{eq:use average on J} uses~\eqref{eq:upper assumption on J}.

The quantity in~\eqref{eq:sum in image}  can be bounded from above using the Lipschitz condition~\eqref{eq:F lipschitz} and the triangle inequality as follows.
\begin{equation}\label{eq:use Lipschitz}
\sum_{j\in J} \left\|F(z_{j,1})-F(z_{j,k_j})\right\|_{W_1(S,d_{G_r(S)})}\le L\sum_{j\in J} d_{G_r(S)}(z_{j,1},z_{j,k_j})\le L\sum_{j\in J}\sum_{i=1}^{k_j-1} d_{G_r(S)}(z_{j,i},z_{j,i+1}).
\end{equation}
Since $\{\{z_{j,i},z_{j,i+1}\}: j\in J\ \wedge\   i\in \{1,\ldots,k_j-1\}\}$ are distinct edges in $E_G$,
\begin{equation}\label{eq:use edge disjoint}
\sum_{j\in J}\sum_{i=1}^{k_j-1} d_{G_r(S)}(z_{j,i},z_{j,i+1})\le \sum_{\{u,v\}\in E_G}d_{G_r(S)}(u,v)\stackrel{\eqref{eq:magnifies edge sum}}{=}\frac{nd}{2}\left(1+\frac{2r|S|}{n}\right).
\end{equation}
A substitution of~\eqref{eq:use edge disjoint} into~\eqref{eq:use Lipschitz}, and contrasting the resulting estimate with~\eqref{eq:use average on J}, yields
\begin{equation}\label{eq:j log j}
\frac{Lnd}{2}\left(1+\frac{2r|S|}{n}\right) \ge \frac{|J|\log |J|}{16\log d}.
\end{equation}
It is elementary to check that if $a\in [1,\infty)$ and $b\in (1,\infty)$ satisfy $a\log a \le b$ then $a\le 2b/\log b$. By applying this with $a=|J|$ and $b= 8Lnd\log d(1+2r|S|/n)\ge n$, we see that~\eqref{eq:j log j} implies that
$$
|J|\le \frac{16nd\log d}{\log n}\left(1+\frac{2r|S|}{n}\right),
$$
thus completing the proof of~\eqref{eq:J upper}.
\end{proof}

\begin{corollary}\label{coro:B small}
Suppose that the following conditions hold true.
\begin{equation}\label{eq:L upper assumption}
d^{16(s-r)}\le \frac{\phi |S|}{8}\qquad\mathrm{and}\qquad L\le  \frac{\phi|S|\log n}{128\left(1+\frac{2r|S|}{n}\right)nd\log d}.
\end{equation}
Then
\begin{equation}\label{eq:max x}
\max_{x\in V_G}|\B_s(x)|< \frac{|S|}{2}.
\end{equation}
\end{corollary}

\begin{proof}
Fix $x\in V_G$. If $\B_s(x)\cap S\neq \emptyset$ then it follows from~\eqref{eq:average distance} that there exist $y,z\in \B_s(x)\cap S$ with
\begin{equation}\label{eq:diameter of intersection big}
d_G(y,z)\ge \frac{\log \left|\B_s(x)\cap S\right|}{4\log d}.
\end{equation}
At the same time, since $y,z\in S$ and $F|_S=f$ is an isometry on $(S,d_{G_r(S)})$, we know that
\begin{multline}\label{eq:diameter of Bs small}
 d_G(y,z)+2r\stackrel{\eqref{eq:def XrS}}{=}d_{G_r(S)}(y,z)=\left\|F(y)-F(z)\right\|_{W_1(S,d_{G_r(S)})}\\\le \left\|F(y)-F(x)\right\|_{W_1(S,d_{G_r(S)})}+\left\|F(x)-F(z)\right\|_{W_1(S,d_{G_r(S)})}\le 2s,
\end{multline}
where in the last step of~\eqref{eq:diameter of Bs small} we used the fact that $y,z\in \B_s(x)$. Contrasting~\eqref{eq:diameter of intersection big} and~\eqref{eq:diameter of Bs small} yields
\begin{equation}\label{eq:intersection size}
\left|\B_s(x)\cap S\right|\le d^{8(s-r)}\le \sqrt{\frac{\phi|S|}{8}}\le  \frac{2|S|}{5},
\end{equation}
where we used the first assumption in~\eqref{eq:L upper assumption} (and that $\phi\le 1\le |S|$). It follows from~\eqref{eq:intersection size} that $|S\setminus \B_s(x)|\ge 3|S|/5$. By combining~\eqref{eq:J lower} and~\eqref{eq:J upper} with this lower bound on $|S\setminus \B_s(x)|$ we see that
\begin{equation}\label{eq:min upper}
\min\left\{\frac{3|S|}{5},\left|\B_s(x)\right|\right\}< \max\left\{\frac{ 4 d^{16(s-r)}}{\phi},\frac{64Lnd\log d}{\phi\log n}\left(1+\frac{2r|S|}{n}\right)\right\}.
\end{equation}
But, the two assumptions in~\eqref{eq:L upper assumption} imply that $3|S|/5$ is greater than the right hand side of~\eqref{eq:min upper}, so
\begin{equation*}
\left|\B_s(x)\right|\le \max\left\{\frac{ 4 d^{16(s-r)}}{\phi},\frac{64Lnd\log d}{\phi\log n}\left(1+\frac{2r|S|}{n}\right)\right\}\stackrel{\eqref{eq:L upper assumption}}{\le} \frac{|S|}{2}.\qedhere
\end{equation*}
\end{proof}

\begin{corollary}\label{coro:second case}
If the conditions  in~\eqref{eq:L upper assumption} are satisfied then
\begin{equation}\label{eq:second L lower}
L\ge \frac{\phi s}{2\left(1+\frac{\diam(G,d_G)}{2r}\right)\left(1+\frac{2r|S|}{n}\right)}.
\end{equation}
\end{corollary}

\begin{proof}
For every $x\in V_G$ and $y\in V_G\setminus \B_s(x)$ we have $\|F(x)-F(y)\|_{W_1(S,d_{G_r(S)})}>s$. Hence,
 \begin{multline}\label{eq:average is at least s}
 \frac{1}{n^2}\sum_{x,y\in V_G}\|F(x)-F(y)\|_{W_1(S,d_{G_r(S)})}\ge \frac{1}{n^2}\sum_{x\in V_{G}} \sum_{y\in V_G\setminus \B_s(x)} \|F(x)-F(y)\|_{W_1(S,d_{G_r(S)})}\\\ge \frac{s}{n^2}\sum_{x\in V_{G}}\left(n-\left|\B_s(x)\right|\right)\ge s\left(1-\frac{\max_{x\in V_G}\left|\B_s(x)\right|}{n}\right)\stackrel{\eqref{eq:max x}}{\ge} \frac{s}{2}.
 \end{multline}
 At the same time, by Lemma~\ref{lem:poincare into wasserstein} we have
 \begin{multline}\label{eq:use poincare}
 \frac{1}{n^2}\sum_{x,y\in V_G}\|F(x)-F(y)\|_{W_1(S,d_{G_r(S)})}\le \frac{2r+\diam(S,d_G)}{(2r+1)\phi |E_G|}\sum_{\{x,y\}\in E_G} \|F(x)-F(y)\|_{W_1(S,d_{G_r(S)})}\\
  \stackrel{\eqref{eq:F lipschitz}}{\le} \frac{L\left(1+\frac{\diam(G,d_G)}{2r}\right)}{\phi|E_G|} \sum_{\{x,y\}\in E_G} d_{G_r(S)}(x,y)
  \stackrel{\eqref{eq:magnifies edge sum}}{=}\frac{L\left(1+\frac{\diam(G,d_G)}{2r}\right)\left(1+\frac{2r|S|}{n}\right)}{\phi}.
 \end{multline}
 The desired estimate~\eqref{eq:second L lower} follows by contrasting~\eqref{eq:average is at least s} with~\eqref{eq:use poincare}.
\end{proof}

\begin{theorem}\label{thm:all range}
Continuing with the same notation as above, if $0<r\le \diam(G,d_G)$ then
\begin{equation}\label{eq:L lower minimum}
L\gtrsim \frac{\phi}{1+\frac{r|S|}{n}}\cdot \min\left\{\frac{|S|\log n}{nd\log d},\frac{16r^2\log d+r\log(\phi|S|/8)}{\diam(G,d_G)\log d}\right\}.
\end{equation}
\end{theorem}

\begin{proof}
If $16r\log d+\log(\phi|S|/8)\le 0$ then~\eqref{eq:L lower minimum} is vacuous. Supposing that $16r\log d+\log(\phi|S|/8)>0$,  choose $s=r+(\log(\phi|S|/8))/(16 \log d)$, so $s>0$ and $d^{16(s-r)}= \phi|S|/8$. The first inequality of~\eqref{eq:L upper assumption} is therefore satisfied, so either the second inequality in~\eqref{eq:L upper assumption} fails, corresponding to a lower bound on $L$, or by Corollary~\ref{coro:second case} the lower bound on $L$ in~\eqref{eq:second L lower} is satisfied. This simplifies to give~\eqref{eq:L lower minimum}.
\end{proof}

For a wide range of the parameters (i.e., $n ,d, \phi, \diam(G,d_G), |S|, r$), Theorem~\ref{thm:all range} yields a nontrivial lower bound on the Lipschitz constant of any mapping $F$ that extends the $1$-Lipschitz function $f$ given in~\eqref{eq:def f wass}. Rather than treating the general case, we shall now proceed to compute an optimal setting of the parameters in~\eqref{eq:L lower minimum} when $G$ has the property that $\diam(G,d_G)\asymp (\log n)/\log d$ and $\phi(G)\asymp  1$. Such graphs exist for any degree $d\ge 3$ and arbitrarily large $n\in \N$, and in fact this holds true with probability tending to $1$ as $n\to \infty$ when $G$ is chosen uniformly at random from the finite set of all $n$-vertex $d$-regular graphs; for the diameter of random regular graphs see e.g.~\cite{Bol82} and for the expansion of random regular graphs see e.g.~\cite[Theorem~4.16]{HLW06}.

So, if $\phi\asymp 1$ and $\diam(G,d_G)\asymp (\log n)/\log d$, and continuing with the assumption  of Theorem~\ref{thm:all range} that $0<r\le \diam(G,d_G)$, the lower bound~\eqref{eq:L lower minimum} becomes
$$
L\gtrsim \frac{1}{1+\frac{r|S|}{n}}\cdot \min\left\{\frac{|S|\log n}{nd\log d},\frac{r\left(r\log d+\log|S|\right)}{\log n}\right\}.
$$
The optimal choice (up to constant factors) is  to take $S\subset V_G$ with $|S|=\lfloor n\sqrt{d\log d}/\sqrt{\log n}\rfloor$ (this is allowed provided $n$ is large enough so as to ensure that $|S|\le n$, specifically it suffices to assume that $n\ge d^d$) and $r\asymp \sqrt{\log n}/\sqrt{d\log d}$. With these choices we see that $L\gtrsim \sqrt{\log n}/\sqrt{d\log d}$. Thus,
\begin{equation}\label{eq:e of expander}
e_{\lfloor n\sqrt{d\log d}/\sqrt{\log n}\rfloor}\left(G_r(S),W_1\left(S,d_{G_r(S)}\right)\right)\gtrsim \frac{\sqrt{\log n}}{\sqrt{d\log d}}.
\end{equation}
For fixed $d$ (say, $d=4$), the estimate~\eqref{eq:e of expander} implies Theorem~\ref{thm:main}.\qed

\begin{remark}\label{rem:aspect}
{\em The above reasoning also provides a new example showing that $\Ae(\e)\gtrsim 1/\e$. Recalling~\eqref{eq:aspect in magnification}, every distinct $x,y\in S$ satisfy $d_{G_r(S)}(x,y)\ge \e\diam(S,d_{G_r(S)})$ where $\e\asymp r/\diam(G,d_G)$. When  $\diam(G,d_G)\asymp (\log n)/\log d$ and $r\asymp \sqrt{\log n}/\sqrt{d\log d}$, this becomes $\e\asymp \sqrt{\log d}/\sqrt{d\log n}$. The lower bound $L\gtrsim \sqrt{\log n}/\sqrt{d\log d}$ therefore becomes $L\gtrsim 1/(\e d)$.}
\end{remark}

\section{Proof of Theorem~\ref{thm:holder}}\label{sec:holder proof}

In Section~\ref{sec:holder intro} we defined the parameters $e^\alpha(X,Z)$ and $e^\alpha_n(X,Z)$, for every $\alpha\in (0,1]$, $n\in \N$ and every two metric spaces $(X,d_X)$ and $(Z,d_Z)$. Below it will be convenient to also use the analogous notation $e^\alpha(X,S,Z)$ for $S\subset X$, i.e.,
$$
e^\alpha(X,S,Z)\eqdef e\big((X,d_X^\alpha),S,(Z,d_Z)\big).
$$

Theorem~\ref{thm:holder} asserts the validity of the following four lower bounds.

\begin{equation}\label{eq:first part of holder}
e^\alpha(\ell_\infty,\ell_2^n)\gtrsim n^{\frac{2\alpha-1}{4\alpha}}\qquad \mathrm{and}\qquad e_n^\alpha(\ell_\infty,\ell_2)\gtrsim (\log n)^{\frac{2\alpha-1}{4\alpha}},
\end{equation}
and
\begin{equation}\label{eq:second part of holder}
e^\alpha(\ell_\infty,\ell_2^n)\gtrsim n^{\alpha^2-\frac12}\qquad \mathrm{and}\qquad e_n^\alpha(\ell_\infty,\ell_2)\gtrsim \left(\frac{\log n}{\log\log n}\right)^{\alpha^2-\frac12}.
\end{equation}
The proof of~\eqref{eq:first part of holder} appears in Section~\ref{sec:twisted} below, and the proof of~\eqref{eq:second part of holder} appears in Section~\ref{sec:linearization} below.

\subsection{Twisted unions of hypercubes}\label{sec:twisted} For $n\in \N$ we identify the discrete hypercube $\{0,1\}^n$ with $\F_2^n$, the vector space of dimension $n$ over the field of size two $\F_2$. We let $e_1,\ldots, e_n$ denote the standard basis of $\F_2^n$ and write $e=e_1+\ldots+e_n$.

\begin{lemma}\label{lem:d}
Suppose that $\alpha\in (\frac12,1]$ and $r,s\in (0,\infty)$ satisfy
\begin{equation}\label{eq:rs assumption}
(2\alpha)^{2\alpha}s(2r)^{2\alpha-1}\ge\left((2\alpha)^{\frac{2\alpha}{2\alpha-1}}-1\right)^{2\alpha-1}.
\end{equation}
Define for every $(x,i),(y,j)\in \F_2^n\times\F_2$,
\begin{equation}\label{eq:def d}
 d((x,i),(y,j))\eqdef \left\{\begin{array}{ll}
\min\left\{s\|x-y\|_1,2r+\|x-y\|_1^{\frac{1}{2\alpha}}\right\}&\mathrm{if}\ i=j=1,\\
\|x-y\|_1^{\frac{1}{2\alpha}}&\mathrm{if}\ i=j=0,\\
r+\min\left\{s\|x-y\|_1,\|x-y\|_1^{\frac{1}{2\alpha}}\right\}&\mathrm{if}\ i\neq j.\end{array}\right.
\end{equation}
Then $d$ is a metric on $\F_2^n\times\F_2$.
\end{lemma}

\begin{proof}
Recall that if $\omega:[0,\infty)\to [0,\infty)$ is concave, nondecreasing, and satisfies $\omega(0)=0$, then $\omega\circ d_X$ is a metric on $X$ for every metric space $(X,d_X)$. An application of this observation to the mapping $t\mapsto t^{1/(2\alpha)}$, which is concave because $\alpha>\frac12$, shows that $d$ is a metric on $\F_2^n\times \{0\}$. Also, since both of the mappings $t\mapsto st$ and $t\mapsto 2r+t^{1/(2\alpha)}$ are concave and increasing, the mapping $t\mapsto \min\{st,2r+t^{1/(2\alpha)}\}$ is concave and increasing. Consequently, $d$ is a metric on $\F_2^n\times \{1\}$.

It therefore remains to show that if $i,j,k\in \{0,1\}$ satisfy $\{i,j,k\}=\{0,1\}$ and $x,y,z\in \F_2^n$ then $d((x,i),(y,j))\le d((x,i),(z,k))+d((z,k),(y,j))$. By translation and permutation of the coordinates we may assume that $x=0$ and $y=e_1+\ldots+e_a$ for some $a\in \n$. In this case, note that if $z$ has a nonzero entry outside $\{1,\ldots,a\}$ then both $\|z\|_1$ and $\|y-z\|_1$ will decrease if we change that entry of $z$ from $1$ to $0$. By the definition of $d$ in~\eqref{eq:def d}, this shows that $d((x,i),(z,k))+d((z,k),(y,j))$ will decrease if we set all the entries of $z$ that are outside $\{1,\ldots,a\}$ to be equal $0$. It therefore suffices to prove that for every $a\in \n$ and $b\in \{1,\ldots,a-1\}$ we have
\begin{equation}\label{eq:ab}
d((0,i),(e_1+\ldots+e_a,j))\le d((0,i),(e_1+\ldots+e_b,k))+d((e_1+\ldots+e_b,k),(e_1+\ldots+e_a,j)).
\end{equation}
We shall establish the validity of~\eqref{eq:ab} through the following case analysis.

\medskip
\noindent{\bf Case 1:} $i=j=0$. In this case necessarily $k=1$, so by~\eqref{eq:def d} the desired inequality~\eqref{eq:ab} becomes
\begin{equation}\label{eq:case 1}
a^{\frac{1}{2\alpha}}\le 2r+\min\left\{sb,b^{\frac{1}{2\alpha}}\right\}+\min\left\{s(a-b),(a-b)^{\frac{1}{2\alpha}}\right\}.
\end{equation}
The mapping $\omega:[0,\infty)\to [0,\infty)$ given by $\omega(t)\eqdef\min\{st,t^{1/(2\alpha)}\}$ is concave, increasing and satisfies $\omega(0)=0$. Consequently, $\omega(b)+\omega(a-b)\ge \omega(a)$. So, in order to prove~\eqref{eq:case 1} it suffices to show that $a^{1/(2\alpha)}\le 2r+\omega(a)=2r+\min\{sa,a^{1/(2\alpha)}\}$. We therefore need to show that $a^{1/(2\alpha)}\le 2r+sa$. Define $\phi:[0,\infty)\to [0,\infty)$ by $\phi(t)=2r+st-t^{1/(2\alpha)}$. The minimum of $\phi$ on $[0,\infty)$ is attained at
$t_{\min}=1/(2\alpha s)^{2\alpha/(2\alpha-1)}$. Hence, by our assumption~\eqref{eq:rs assumption} we have
$$
\phi(a)\ge \phi(t_{\min})=2r+\frac{s}{(2\alpha s)^{\frac{2\alpha}{2\alpha-1}}}-\frac{1}{(2\alpha s)^{\frac{1}{2\alpha-1}}}=2r-\frac{(2\alpha)^{\frac{2\alpha}{2\alpha-1}}-1}{((2\alpha)^{2\alpha}s)^{\frac{1}{2\alpha-1}}}\stackrel{\eqref{eq:rs assumption}}{\ge} 0.
$$

\medskip
\noindent{\bf Case 2:} $i=j=1$. In this case necessarily $k=0$, so by~\eqref{eq:def d} the desired inequality~\eqref{eq:ab} becomes
\begin{equation}\label{eq:case 2}
\min\left\{sa,2r+a^{\frac{1}{2\alpha}}\right\}\le 2r+\min\left\{sb,b^{\frac{1}{2\alpha}}\right\}+\min\left\{s(a-b),(a-b)^{\frac{1}{2\alpha}}\right\}.
\end{equation}
As in Case 1, by the subadditivity of $\omega$ we see that in order to prove~\eqref{eq:case 2} it suffices to show that $\min\{sa,2r+a^{1/(2\alpha)}\}\le 2r+\omega(a) =2r+\min\{sa,a^{1/(2\alpha)}\}$, which is immediate.

\medskip
\noindent{\bf Case 3:} $\{i,j\}=\{0,1\}$. In this case, by interchanging the roles of $b$ and $a-b$, if necessary, in order to prove the desired inequality~\eqref{eq:ab} one must show that the following two inequalities hold true
\begin{equation*}
r+\min\left\{sa,a^{\frac{1}{2\alpha}}\right\}\le
b^{\frac{1}{2\alpha}}+r+\min\left\{s(a-b),(a-b)^{\frac{1}{2\alpha}}\right\},
\end{equation*}
and
\begin{equation*}
r+\min\left\{sa,a^{\frac{1}{2\alpha}}\right\}\le
r+\min\left\{sb,b^{\frac{1}{2\alpha}}\right\}+\min\left\{s(a-b),2r+(a-b)^{\frac{1}{2\alpha}}\right\}.
\end{equation*}
Thus, we need to check that
\begin{equation}\label{eq:case 3}
\min\left\{sa,a^{\frac{1}{2\alpha}}\right\}\le
b^{\frac{1}{2\alpha}}+\min\left\{s(a-b),(a-b)^{\frac{1}{2\alpha}}\right\},
\end{equation}
and
\begin{equation}\label{eq:case 3'}
\min\left\{sa,a^{\frac{1}{2\alpha}}\right\}\le
\min\left\{sb,b^{\frac{1}{2\alpha}}\right\}+\min\left\{s(a-b),2r+(a-b)^{\frac{1}{2\alpha}}\right\}.
\end{equation}
Inequality~\eqref{eq:case 3'} follows immediately from the subadditivity of $\omega$ as follows.
\begin{equation*}
\min\left\{sb,b^{\frac{1}{2\alpha}}\right\}+\min\left\{s(a-b),2r+(a-b)^{\frac{1}{2\alpha}}\right\}\ge \omega(b)+\omega(a-b)\ge \omega(a)= \min\left\{sa,a^{\frac{1}{2\alpha}}\right\}.
\end{equation*}
Since the mapping $t\mapsto t^{1/(2\alpha)}$ is subadditive, we have $b^{1/(2\alpha)}+(a-b)^{1/(2\alpha)}\ge a^{1/(2\alpha)}$. Therefore, in order to prove~\eqref{eq:case 3} it suffices to show that $b^{1/(2\alpha)}+s(a-b)\ge \min\{sa,a^{1/(2\alpha)}\}$. Define $\psi:[0,a]\to \R$ by $\psi(t)=t^{1/(2\alpha)}+s(a-t)$. Then $\psi$ is concave. It follows that  the minimum of $\psi$ on the interval $[0,a]$ is attained at one of its endpoints, and therefore $\psi(b)\ge \min\{\psi(0),\psi(a)\}=\min\{sa,a^{1/(2\alpha)}\}$.
\end{proof}

The proof of the following theorem is a variant of an argument of~\cite{JLS86}.

\begin{theorem}\label{thm:L lower} Continuing with the notation of Lemma~\ref{lem:d} and assuming from now on that the condition~\eqref{eq:rs assumption} is satisfied, define $f:\F_2^n\times \{0\}\to  \ell_2^n$ by setting $f(x,0)=x$ for every $x\in \F_2^n$. Then $\|f(x,0)-f(y,0)\|_2= d((x,0),(y,0))^\alpha$ for every $x,y\in \F_2^n$. At the same time, if $F:\F_2^n\times \F_2\to \ell_2^n$ extends $f$ and satisfies $\|F(u)-F(v)\|_2\le Ld(u,v)^\alpha$ for some $L\in (0,\infty)$ and every $u,v\in \F_2^n\times \F_2$ then we necessarily have
\begin{equation}\label{eq:L lower}
L\ge \frac{\sqrt{n}}{s^\alpha\sqrt{n}+2r^\alpha}.
\end{equation}
\end{theorem}

\begin{proof}
Note that for every $x,y\in \F_2^n$ we have $$\|f(x,0)-f(y,0)\|_2=\|x-y\|_2=\sqrt{\|x-y\|_1}=d((x,0),(y,0))^\alpha,$$ by the definition of $d$ in~\eqref{eq:def d}. Next, the assumptions on $F$ imply that for every $x,y\in \F_2^n$ and every $j \in \n$ we have
$$
\|F(x+e_j,1)-F(x,1)\|_2\le Ld((x+e_j,1),(x,1))^\alpha\stackrel{\eqref{eq:def d}}{=} L\min\left\{s^\alpha,(2r+1)^\alpha\right\}\le Ls^\alpha.
$$
Hence,
\begin{equation}\label{eq:cube edges}
\sum_{j=1}^n\sum_{x\in \F_2^n} \|F(x+e_j,1)-F(x,1)\|_2^2\le n2^{n}L^2s^{2\alpha}.
\end{equation}
Also, the following estimate holds true for every $x\in \F_2^n$.
\begin{align*}
\|F(x+e,1)&-F(x,1)\|_2\\&\ge \|F(x+e,0)-F(x,0)\|_2-\|F(x+e,1)-F(x+e,0)\|_2-\|F(x,1)-F(x,0)\|_2\\
&\ge \|f(x+e,0)-f(x,0)\|_2-Ld((x+e,1),(x+e,0))^\alpha-Ld((x,1),(x,0))^\alpha\\
&=\sqrt{n}-2Lr^\alpha.
\end{align*}
Hence,
\begin{equation}\label{eq:cube diagonals}
\sum_{x\in \F_2^n}\|F(x+e,1)-F(x,1)\|_2^2\ge 2^n\left(\max\left\{\sqrt{n}-2Lr^\alpha,0\right\}\right)^2.
\end{equation}
By a classical inequality of Enflo~\cite{Enf69} (see also~\cite[\S~15.4]{Mat02}),
\begin{equation}\label{eq:enflo}
\sum_{x\in \F_2^n}\|F(x+e,1)-F(x,1)\|_2^2\le \sum_{j=1}^n\sum_{x\in \F_2^n} \|F(x+e_j,1)-F(x,1)\|_2^2.
\end{equation}
A substitution of~\eqref{eq:cube edges} and~\eqref{eq:cube diagonals} into~\eqref{eq:enflo} now yields the estimate
\begin{equation*}
\sqrt{n}-2Lr^\alpha\le Ls^\alpha\sqrt{n}\implies L\ge \frac{\sqrt{n}}{s^\alpha\sqrt{n}+2r^\alpha}. \qedhere
\end{equation*}
\end{proof}

One should clearly choose those $r,s\in (0,\infty)$ that satisfy the constraint~\eqref{eq:rs assumption} and maximize the right hand side of~\eqref{eq:rs assumption}. This yields better dependence (in terms of an $\alpha$-dependent constant factor but not in term of dependence on $n$) than the following sub-optimal choices.
\begin{equation}\label{eq:rs choices}
r\eqdef n^{\frac{1}{4\alpha^2}}\qquad\mathrm{and}\qquad s\eqdef n^{-\frac{2\alpha-1}{4\alpha^2}}.
\end{equation}
It is elementary to check that under these choices the constraint~\eqref{eq:rs assumption} is satisfied, yielding the estimate $L\gtrsim n^{(2\alpha-1)/(4\alpha)}$ in Theorem~\ref{thm:L lower}. This implies the validity of~\eqref{eq:first part of holder}, i.e., the first half of Theorem~\ref{thm:holder}. Namely, the following bounds hold true.

\begin{corollary}\label{coro:4}
Continuing with the notation of Lemma~\ref{lem:d}, we have
$$
e^\alpha(\ell_\infty,\ell_2^n)\ge e^\alpha\big((\F_2^n\times \F_2^n,d),\F_2^n\times \{0\},\ell_2^n\big)\gtrsim n^{\frac{2\alpha-1}{4\alpha}}.
$$
Also, since $n\asymp \log |\F_2^n|$, it follows that for arbitrarily large $N\in \N$ we have
$$
e^\alpha_N(\ell_\infty,\ell_2)\gtrsim (\log N)^{\frac{2\alpha-1}{4\alpha}}.
$$
\end{corollary}

\begin{remark}
{\em Consider the mappings $\omega_0,\omega_1:[0,\infty)\to [0,\infty)$ given by
$$
\forall\, t\ge 0,\qquad \omega_0(t)=t^{1/(2\alpha)}\qquad \mathrm{and} \qquad \omega_1(t)=\min\{st,2r+\omega_0(t)\}.
$$
Both $\omega_0$ and $\omega_1$ are concave, increasing and vanish at the origin. The metric $d$ given in~\eqref{eq:def d} satisfies $d((x,0),d(y,0))=\omega_0(\|x-y\|_1)$ and $d((x,1),(y,1))=\omega_1(\|x-y\|_1)$. As explained in Remark 5.5 of~\cite{MN-duke}, it follows that both of the metric spaces $(\F_2^n\times \{0\},d)$ and $(\F_2^n\times \{1\},d)$ embed into $\ell_1$ with $O(1)$ distortion. We do not know whether the metric space $(\F_2^n\times \F_2,d)$ admits an embedding into $\ell_1$ with $O(1)$ distortion. If it were the case that any embedding of $(\F_2^n\times \F_2,d)$ (say, when $\alpha=1$) into $\ell_1$ must incur bi-Lipschitz distortion that tends to $\infty$ as $n\to\infty$, this would be the first known example of a metric space that can be partitioned into two subsets, each of which well-embeds into $\ell_1$ yet the entire space does not.  For more on such questions, see~\cite{MN-skeletons}.  }
\end{remark}

\begin{remark}\label{rem:general unions}
{\em The metric $d$ in~\eqref{eq:def d}, as well as the magnification of a metric space that was described in Section~\ref{sec:magnification}, are both special cases of a generalization of a construction that was used in~\cite{JLS86} for the purpose of proving a Lipschitz-nonextendability result. Variants of this construction were also used in~\cite{Lan99,CKR04}. The general idea is the following procedure to ``glue" two metric spaces. Suppose that $(X,d_X)$ and $(Y,d_Y)$ are finite metric spaces with $X$ and $Y$ disjoint as sets. Suppose also that we are given a mapping $\sigma:X\to Y$ and $r\in (0,\infty)$. Define a weighted graph structure on $X\cup Y$ as follows. If $x_1,x_2\in X$ then join $x_1$ and $x_2$ by an edge of weight $d_X(x_1,x_2)$. Similarly, if $y_1,y_2\in Y$ then join $y_1$ and $y_2$ by an edge of weight $d_Y(y_1,y_2)$. Also, for every $x\in X$ join $x$ and $\sigma(x)$ by an edge of weight $r$ (a further generalization of this procedure could allow the weight of the edge $\{x,\sigma(x)\}$ to depend on $x$). Consider now the shortest-path metric that  this weighted graph induces on $X\cup Y$. One can check that the above metric spaces can also be described as subsets of this general construction, whose usefulness is probably yet to be fully exploited.
}
\end{remark}

\subsection{Linearization}\label{sec:linearization} Our goal here is to prove~\eqref{eq:second part of holder}, thus completing the proof of Theorem~\ref{thm:holder}. The argument below is based on the linearization procedure of~\cite{JL84}, with several modifications that yield quantitative improvements and, more importantly, allow us to to treat H\"older mappings.

In what follows, given a Banach space $(W,\|\cdot\|_W)$ we denote
$$
B_W\eqdef \{x\in W:\ \|x\|_W\le 1\}\qquad\mathrm{and}\qquad S_W\eqdef \partial B_W=\{x\in W:\ \|x\|_W= 1\}.
$$
If $V\subset W$ is a linear subspace then we will always consider it as being equipped with the norm inherited from $W$, thus slightly abusing notation by denoting the sets $V\cap B_W$ and $V\cap S_W$ by $B_V$ and $S_V$, respectively.   A mapping $h:W_1\to W_2$ between two real vector spaces $W_1,W_2$ is said to be positively homogeneous if $h(\lambda x)=\lambda h(x)$ for every $\lambda\in [0,\infty)$.

Lemma~\ref{lem:sqrt m} below is a variant of Lemma~5 in~\cite{JL84}, with the difference being that in~\cite{JL84} the conclusion corresponding to~\eqref{eq:root m bound a bit better} is weaker in the sense that the factor $\sqrt{m}D$ is replaced by $m$. The difference between our proof and the proof in~\cite{JL84} is that we perform averaging in the image of $T$ rather than in $X$ itself, allowing for a more refined estimate that we shall use later.

\begin{lemma}\label{lem:sqrt m}
Fix $m\in \N$ and $D\in [1,\infty)$. Let $(Y,\|\cdot\|_Y)$ and $(Z,\|\cdot\|_Z)$ be Banach spaces and $X\subset Y$ an $m$-dimensional linear subspace of $Y$. Suppose that  $T:X\to \ell_2^m$ is a linear operator
that satisfies
\begin{equation}\label{eq:T assumption}
\forall\, x\in X,\qquad \|x\|_Y\le \|Tx\|_2\le D\|x\|_Y.
\end{equation}
Suppose also that we are given a positively homogeneous Lipschitz mapping $\psi: X\to Z$ and a linear operator $U:X\to Z$. Then there exists a positively homogeneous mapping $\Psi:Y\to Z$ that satisfies
\begin{equation}\label{eq:root m bound a bit better}
\|\Psi\|_{\Lip}\lesssim \|\psi\|_{\Lip}\qquad\mathrm{and}\qquad \|\Psi|_X-U\|_{\Lip}\lesssim \sqrt{m}D\sup_{z\in S_X}\|\psi(z)-Uz\|_Z.
\end{equation}
\end{lemma}

\begin{proof} Write $B_2^m=B_{\ell_2^m}$ and
\begin{equation}\label{eq:volume formula}
v_m=\vol\left(B_2^m\right)=\frac{\pi^{\frac{m}{2}}}{\Gamma\left(1+\frac{m}{2}\right)}.
\end{equation}

Define $\phi:Y\to Z$ as follows.
$$
\forall\, y\in Y,\qquad \phi(y)\eqdef \frac{1}{v_m}\int_{B_2^m} \psi\left(y+T^{-1}w\right)dw=\frac{1}{v_m}\int_{Ty+B_2^m} \psi\left(T^{-1}w\right)dw.
$$
Then, because $\phi$ is obtained by applying an averaging operator to $\psi$, we have
\begin{equation}\label{eq:phi lip}
\|\phi\|_{\Lip}\le \|\psi\|_{\Lip}.
\end{equation}
 Also, because $\psi(0)=0$ (recall that $\psi$ is positively homogeneous), we have
\begin{equation}\label{eq:sup norm phi}
\sup_{y\in S_Y}\|\phi(y)\|_Z\le \|\psi\|_{\Lip}\sup_{(y,w)\in  S_Y\times B_2^m}\|y+T^{-1}w\|_Y\stackrel{\eqref{eq:T assumption}}{\le} 2\|\psi\|_{\Lip}.
\end{equation}
Similarly,
\begin{multline}\label{eq:phi-U sup}
\sup_{y\in S_X}\|\phi(y)-Uy\|_Z =\sup_{y\in S_X}\Big\|\frac{1}{v_m}\int_{B_2^m} (\psi-U)\left(y+T^{-1}w\right)dw\Big\|_Z\\
\le \sup_{(y,w)\in  S_X\times B_2^m}\|(\psi-U)(y+T^{-1}w)\|_Z\le \sup_{z\in 2B_X} \|\psi(z)-Uz\|_Z=2\sup_{z\in S_X} \|\psi(z)-Uz\|_Z,
\end{multline}
where the penultimate step in~\eqref{eq:phi-U sup} holds true because, due to~\eqref{eq:T assumption}, we have $T^{-1}B_2^m\subset B_X$, so $S_X+T^{-1}B_2^m\subset 2B_X$. The final step of~\eqref{eq:phi-U sup} holds true because $\psi-U$ is positively homogeneous.

Fix distinct $x,y\in S_X$ and write $2r\eqdef \|x-y\|_Y$.  Then, since $U$ and $T$ are linear,
\begin{multline*}
(\phi-U)(x)-(\phi-U)(y)\\=\frac{1}{v_m}\int_{(Tx+B_2^m)\setminus (Ty+B_2^m)} (\psi-U)(T^{-1}(w))dw-\frac{1}{v_n}\int_{(Ty+B_2^m)\setminus (Tx+B_2^n)} (\psi-U)(T^{-1}(w))dw.
\end{multline*}
Consequently,
\begin{multline}\label{eq:symmetric difference appears}
\|(\phi-U)(x)-(\phi-U)(y)\|_Z\le \frac{1}{v_m}\int_{(Tx+B_2^m)\triangle (Ty+B_2^m)} \|(\psi-U)(T^{-1}(w))\|_Zdw
\\\le \frac{\vol\left((Tx+B_2^m)\triangle (Ty+B_2^m)\right)}{v_n}\cdot \sup_{z\in (x+T^{-1}B_2^m)\triangle (y+T^{-1}B_2^m)}\|\psi(z)-Uz\|_Z.
\end{multline}

We note that
\begin{equation}\label{eq:volume of union}
\vol\left((Tx+B_2^m)\cup (Ty+B_2^m)\right)\le v_m+\|Tx-Ty\|_2v_{m-1}.
\end{equation}
Indeed, by rotation invariance, it suffices to verify~\eqref{eq:volume of union} when $Tx=r e_1$ and $Ty=-r e_1$, where $e_1,\ldots,e_m$ is the standard basis of $\R^m$ and we recall that $\|Tx-Ty\|_2=2r$.
Identifying $\R^m$ with $\R\times \R^{m-1}$, the estimate~\eqref{eq:volume of union} is a consequence of the fact that  $(re_1+B_2^m)\cup(-re_1+B_2^m)$ is contained in the union of the three sets $([r,\infty)\times \R^{m-1})\cap (re_1+B_2^m)$, $((-\infty,-r]\times \R^{m-1})\cap (-re_1+B_2^m)$ and $[-r,r]\times B_2^{m-1}$, whose interiors are disjoint.

Now, due to~\eqref{eq:volume of union} we have
\begin{align*}
\nonumber \vol\left((Tx+B_2^m)\triangle (Ty+B_2^m)\right)&=2\vol\left((Tx+B_2^m)\cup (Ty+B_2^m)\right)-\vol\left(Tx+B_2^m\right)-\vol\left(Ty+B_2^m\right)\\&\le 2\|Tx-Ty\|_2v_{m-1}.
\end{align*}
So, the first term in the right hand side of~\eqref{eq:symmetric difference appears} can be bounded as follows.
\begin{equation}\label{eq:symmetric bound}
\frac{\vol\left((Tx+B_2^m)\triangle (Ty+B_2^m)\right)}{v_n}\le \frac{2v_{m-1}}{v_m} \|Tx-Ty\|_2\stackrel{\eqref{eq:T assumption}}{\le} \sqrt{m}D\|x-y\|_Y,
\end{equation}
where we used the fact that $2v_{m-1}/v_m\le \sqrt{m}$, which follows from~\eqref{eq:volume formula} and Stirling's formula (in addition, we actually have $\lim_{m\to \infty} 2v_{m-1}/v_m=\sqrt{2/\pi}$).

To bound the second term in the right hand side of~\eqref{eq:symmetric difference appears}, note that $T^{-1}B_2^m\subset B_X$, by~\eqref{eq:T assumption}. Therefore $(x+T^{-1}B_2^m)\triangle (y+T^{-1}B_2^m)\subset 2B_X$, since $x,y\in S_X$. So,
\begin{equation}\label{eq:sup bound}
\sup_{z\in (x+T^{-1}B_2^m)\triangle (y+T^{-1}B_2^m)}\|\psi(z)-Uz\|_Z\le \sup_{z\in 2 B_X}\|\psi(z)-Uz\|_Z=2\sup_{z\in  S_X}\|\psi(z)-Uz\|_Z,
\end{equation}
where we used the fact that $\psi-U$ is positively homogeneous.

A substitution of~\eqref{eq:symmetric bound} and~\eqref{eq:sup bound} into~\eqref{eq:symmetric difference appears} shows that
\begin{equation}\label{eq:better with root of dim}
\|(\phi-U)|_{S_X}\|_{\Lip}\le 4\sqrt{m}D\sup_{z\in S_X}\|\psi(z)-Uz\|_Z.
\end{equation}
Define $\Psi:Y\to Z$ to be the positively homogeneous extension of $\phi|_{S_Y}$, i.e., $\Psi(y)\eqdef \|y\|_Y\phi(y/\|y\|_Y)$ for every $y\in Y\setminus \{0\}$ and $\Psi(0)=0$. Then a simple computation (that is carried out in detail in~\cite[Lemma~2]{JL84}) shows that
$$
\|\Psi\|_{\Lip}\le 2\|\phi|_{S_Y}\|_{\Lip}+\sup_{y\in S_Y}\|\phi(y)\|_Z\stackrel{\eqref{eq:phi lip}\wedge \eqref{eq:sup norm phi}}{\le} 4\|\psi\|_{\Lip}.
$$
For the same reason, since $\Psi-U$ is the positively homogeneous extension of $(\phi-U)|_{S_X}$ to $X$,
\begin{equation*}
\|\Psi|_X-U\|_{\Lip}\le 2\|(\phi-U)|_{S_X}\|_{\Lip}+\sup_{y\in S_X}\|\phi(y)-Uy\|_Z\stackrel{\eqref{eq:phi-U sup} \wedge \eqref{eq:better with root of dim}}{\le} 10\sqrt{m}D\sup_{z\in S_X}\|\psi(z)-Uz\|_Z.\tag*{\qedhere}
\end{equation*}
\end{proof}

We shall use below the following lemma of Begun~\cite{Beg99}; alternatively one could use below an argument of Bourgain~\cite{Bou87} that Begun significantly simplified in~\cite{Beg99}.

\begin{lemma}\label{lem:begun}
Fix $\delta,L,\tau\in (0,\infty)$ and $m\in \N$. Suppose that $(Y,\|\cdot\|_Y)$ is an $m$-dimensional normed space and that $(Z,\|\cdot\|_Z)$ is a Banach space. Let $K\subset Y$ be a convex set and suppose that we are given a mapping $F:K+\tau B_Y\to Z$ that satisfies
$$
\forall\, x,y\in K+\tau B_Y,\qquad \|F(x)-F(y)\|_Z\le L(\|x-y\|_Y+\delta).
$$
Define $h:K\to Z$ by
\begin{equation}\label{eq:def h}
\forall\, x\in K,\qquad  h(x)\eqdef \frac{1}{\vol(\tau B_Y)}\int_{\tau B_Y} F(x+y)dy,
\end{equation}
where the integration and $\vol(\cdot)$ are interpreted through an identification of $Y$ with $\R^m$. Then
$$
\|h\|_{\Lip}\le L\left(1+\frac{\delta m}{2\tau}\right).
$$
\end{lemma}

\begin{corollary}\label{coro:holder approx}
Fix $\tau\in (0,\infty)$, $\alpha\in (0,1]$ and $m\in \N$. Suppose that $(Y,\|\cdot\|_Y)$ is an $m$-dimensional normed space and that $(Z,\|\cdot\|_Z)$ is a Banach space. Let $K\subset Y$ be a convex set and suppose that $F:K+\tau B_Y\to Z$ is $\alpha$-H\"older. Then there exists $h:K\to Z$ that satisfies
$$
\sup_{x\in K}\|h(x)-F(x)\|_Z\le \tau^\alpha\|F\|_{\Lip(\alpha)} \qquad\mathrm{and}\qquad \|h\|_{\Lip}\le \frac{\alpha^2(m/\tau)^{1-\alpha}+1}{\alpha}\|F\|_{\Lip(\alpha)}.
$$
\end{corollary}

\begin{proof}
Denote $L=\|F\|_{\Lip(\alpha)}$. Note that for every $\delta\in (0,\infty)$ and every $x,y\in K+\tau B_Y$ we have
$$
\|F(x)-F(y)\|_Z\le L\|x-y\|_Y^\alpha\le L(\|x-y\|_Z+\delta)\sup_{t\ge 0} \frac{t^\alpha}{t+\delta}=\frac{(1-\alpha)^{1-\alpha}\alpha^\alpha L}{\delta^{1-\alpha}} (\|x-y\|_Z+\delta).
$$
We may therefore apply Lemma~\ref{lem:begun}, deducing that for $h$ defined as in~\eqref{eq:def h} we have
\begin{equation}\label{eq:use begun}
\|h\|_{\Lip}\le \frac{(1-\alpha)^{1-\alpha}\alpha^\alpha L}{\delta^{1-\alpha}} \left(1+\frac{\delta m}{2\tau}\right).
\end{equation}
The value of $\delta$ that minimizes the right hand side of~\eqref{eq:use begun} is $\delta=2\tau (1-\alpha)/(\alpha m)$, which gives
$$
\|h\|_{\Lip}\le (1-\alpha)^{1-\alpha}\alpha^\alpha \left(\frac{\alpha^{1-\alpha}}{(1-\alpha)^{1-\alpha}}\left(\frac{m}{2\tau}\right)^{1-\alpha}+\frac{1-\alpha}{\alpha}\right)L\le \frac{\alpha^2(m/\tau)^{1-\alpha}+1}{\alpha}L.
$$
It remains to note that for every $x\in K$ we have
\begin{equation*}
\|h(x)-F(x)\|_Z=\Big\|\frac{1}{\vol(\tau B_Y)}\int_{\tau B_Y} (F(x+y)-F(x)) dy\Big\|_Z\le  \frac{1}{\vol(\tau B_Y)}\int_{\tau B_Y} L\|y\|_Y^\alpha dy\le L\tau^\alpha. \tag*{\qedhere}
\end{equation*}
\end{proof}

\begin{theorem}\label{thm:linearization}
There is a universal constant $\kappa\in (0,1)$ with the following properties. Fix $\alpha\in (0,1]$, $M,m\in \N$ and $C,D\in [1,\infty)$. Suppose that $(Y,\|\cdot\|_Y)$ is an $M$-dimensional normed space and that $(Z,\|\cdot\|_Z)$ is an $m$-dimensional normed space. Let $X\subset Y$ be an $m$-dimensional subspace of $Y$ whose Banach--Mazur distance to $\ell_2^m$ is at most $D$, i.e., there exists a linear operator $T:X\to \ell_2^m$ satisfying~\eqref{eq:T assumption}. Suppose also that $U:X\to Z$ is a linear operator with
\begin{equation}\label{eq:U assumption}
\forall x\in X,\qquad \|x\|_Y\le \|Ux\|_Z\le C\|x\|_Y.
\end{equation}
For every $\e\in (0,1)$ let $\mathscr{N}_\e$ be an $\e$-net in $S_X$. Then, under the additional assumption that
\begin{equation}\label{eq:epsilon self referential}
\e\le \left(\frac{\kappa}{\sqrt{m}CD e^\alpha(Y,\mathscr{N}_\e,Z)}\right)^{\frac{1}{\alpha}},
\end{equation}
there exists a linear operator $V:Y\to Z$ that extends $U$ and whose operator norm satisfies
\begin{equation}\label{eq:V conclusion}
\|V\|_{Y\to Z}\lesssim \alpha M^{1-\alpha}m^{\frac{1-\alpha}{2\alpha}}C^{\frac{1}{\alpha}}D^{\frac{1-\alpha}{\alpha}}  e^\alpha(Y,\mathscr{N}_\e,Z)^{\frac{1}{\alpha}}+\frac{C}{\alpha}e^\alpha(Y,\mathscr{N}_\e,Z).
\end{equation}
\end{theorem}

\begin{proof}
If $x,y\in B_X$ then $\|x-y\|_Y\le 2$, hence $\|Ux-Uy\|_Z\le C\|x-y\|_Y\le 2^{1-\alpha}C\|x-y\|_Y^\alpha$. This implies that the restriction of $U$ to $\mathscr{N}_\e$ is $\alpha$-H\"older with constant $2C$. Denoting $L=e^\alpha(Y,\mathscr{N}_\e,Z)$, it follows that there exists $g:Y\to Z$ that extends $U|_{\mathscr{N}_\e}$ and is $\alpha$-H\"older with constant $2CL$.

Let $\rho:Z\to CB_Z$ be the radial retraction of $Z$ onto $CB_Z$, i.e., $\rho(z)=z$ if $\|z\|_Z\le C$ and $\rho(z)=Cz/\|z\|_Z$ if $\|z\|_Z>C$. It is straightforward to check that $\rho$ is $2$-Lipschitz. Therefore, if we define $F=\rho\circ g$ then $\|F\|_{\Lip(\alpha)}\le 2\|g\|_{\Lip(\alpha)}\le 4CL$ and $\|F(y)\|_Z\le C$ for every $y\in Y$. Moreover, for every $a\in \mathscr{N}_\e$ we have $\|g(a)\|_Z=\|Ua\|_Z\le C$, so $F(a)=\rho(g(a))=g(a)$. Thus $F:Y\to CB_Z$ is also an extension of $U|_A$. Fix $\tau\in (0,1)$ that will be determined later.  By Corollary~\ref{coro:holder approx} (applied with $K=Y$) there exists $h:Y\to Z$ that satisfies
\begin{equation}\label{eq:hF}
\forall\, y\in Y\qquad \|h(y)-F(y)\|_Z\le 4CL\tau^\alpha,
\end{equation}
and
\begin{equation}\label{eq:h Lip}
\|h\|_{\Lip} \le 4\frac{\alpha^2(M/\tau)^{1-\alpha}+1}{\alpha}CL.
\end{equation}
Then
\begin{equation}\label{eq:h sup}
\sup_{y\in S_Y}\|h(y)\|_Z\stackrel{\eqref{eq:hF}}{\le} \sup_{y\in S_Y}\|F(y)\|_Z+4CL\tau^\alpha\le C+4CL\tau^\alpha,
\end{equation}
where in the last step of~\eqref{eq:h sup} we used the fact that $F$ takes values in $CB_Z$.

Denote by $\psi:Y\to Z$ the positively homogeneous extension of $h|_{S_Y}$ to $Y$. By~\cite[Lemma~2]{JL84},
\begin{equation}\label{eq:psi lip for using lemma}
\|\psi\|_{\Lip}\le 2\|h|_{S_Y}\|_{\Lip}+\sup_{y\in S_Y}\|h(y)\|_Z\stackrel{\eqref{eq:h Lip}\wedge \eqref{eq:h sup}}{\le} 4\frac{\alpha^2(M/\tau)^{1-\alpha}+1}{\alpha}CL+C+4CL\tau^\alpha.
\end{equation}
Take $x\in S_X$ (thus $\psi(x)=h(x)$) and $a\in \mathscr{N}_\e$ with $\|a-x\|_Y\le \e$. Then $h(a)=F(a)=Ua$. Hence,
\begin{multline}\label{eq:psi approx}
\|\psi(x)-Ux\|_Z\le \|h(x)-F(x)\|_Z+\|F(x)-F(a)\|_Z+\|Ua-Ux\|_Z\\ \le 4CL\tau^\alpha+4CL\|x-a\|_Y^\alpha+C\|x-a\|_Y
\le 4CL\tau^\alpha+4CL\e^\alpha+C\e\le 5CL(\tau^\alpha+\e^\alpha).
\end{multline}

Due to~\eqref{eq:psi lip for using lemma} and~\eqref{eq:psi approx}, Lemma~\ref{lem:sqrt m} implies that there exists $\Psi:Y\to Z$ that satisfies
$$
\|\Psi\|_{\Lip}\lesssim \frac{\alpha^2(M/\tau)^{1-\alpha}+1}{\alpha}CL+C+CL\tau^\alpha\lesssim \frac{\alpha^2(M/\tau)^{1-\alpha}+1}{\alpha}CL,
$$
and
$$
\|\Psi|_X-U\|_{\Lip}\lesssim \sqrt{m}CDL(\tau^\alpha+\e^\alpha).
$$

By Proposition~1 in~\cite{JL84}, which relies on an important linearization result due to Lindenstrauss~\cite{Lin64}, there exists a linear operator $S:Y\to Z$ such that
\begin{equation}\label{eq:S norm}
\|S\|_{Y\to Z}\le \|\Psi\|_{\Lip}\lesssim \frac{\alpha^2(M/\tau)^{1-\alpha}+1}{\alpha}CL,
\end{equation}
and
\begin{equation}\label{eq: s-u norm}
\|S|_X-U\|_{X\to Z}\le \|\Psi|_X-U\|_{\Lip}\lesssim \sqrt{m}CDL(\tau^\alpha+\e^\alpha).
\end{equation}
Our assumption~\eqref{eq:epsilon self referential} says that $\sqrt{m}CDL \e^\alpha\le \kappa$. So, if we choose $\tau=(\kappa/(\sqrt{m}CDL))^{1/\alpha}$ and $\kappa>0$ is a small enough universal constant then  it would follow from~\eqref{eq:S norm} and~\eqref{eq: s-u norm} that
\begin{equation}\label{eq:tau chosen}
\|S\|_{Y\to Z}\lesssim \alpha M^{1-\alpha}m^{\frac{1-\alpha}{2\alpha}}C^{\frac{1}{\alpha}}D^{\frac{1-\alpha}{\alpha}}L^{\frac{1}{\alpha}}+\frac{CL}{\alpha}\qquad\mathrm{and}\qquad \|S|_X-U\|_{X\to Z}\le \frac12.
\end{equation}
Letting $I_Z:Z\to Z$ denote the identity mapping on $Z$, we have
$$
\left\|(S|_X)U^{-1}-I_Z\right\|_{Z\to Z}=\left\|(S|_X-U)U^{-1}\right\|_{Z\to Z}\le \|S|_X-U\|_{X\to Z}\|U^{-1}\|_{Z\to X}\stackrel{\eqref{eq:U assumption}\wedge \eqref{eq:tau chosen}}{\le}\frac12.
$$
Consequently, the operator $(S|_X)U^{-1}:Z\to Z$ is invertible and $\|((S|_X)U^{-1})^{-1}\|_{Z\to Z}\le 2$. Define
$$
V\eqdef ((S|_X)U^{-1})^{-1}S:Y\to Z.
$$
Then $V$ extends $S$ and
\begin{equation*}
\|V\|_{Y\to Z}\le \left\|((S|_X)U^{-1})^{-1}\right\|_{Z\to Z}\|S\|_{Y\to Z}\le 2\|S\|_{Y\to Z}\stackrel{\eqref{eq:tau chosen}}{\lesssim} \alpha M^{1-\alpha}m^{\frac{1-\alpha}{2\alpha}}C^{\frac{1}{\alpha}}D^{\frac{1-\alpha}{\alpha}}L^{\frac{1}{\alpha}}+\frac{CL}{\alpha}. \tag*{\qedhere}
\end{equation*}
\end{proof}

Given a normed space $(Y,\|\cdot\|_Y)$ and a linear subspace $X\subset Y$, the projection constant of $X$ relative to $Y$ is denoted $\lambda(X,Y)$. Thus $\lambda(X,Y)$ is the infimum over those $\lambda\in [1,\infty]$ for which there exists a linear projection $P$ from $Y$ onto $X$ with $\|P\|_{Y\to X}\le \lambda$.

\begin{corollary}\label{coro:projection constant}
Let $\kappa$ be the universal constant from Theorem~\ref{thm:linearization}. Fix $\alpha\in (0,1]$, $M,m\in \N$ and $D\in [1,\infty)$.  Let $X\subset Y$ be an $m$-dimensional subspace of $Y$ whose Banach--Mazur distance to $\ell_2^m$ is at most $D$. For every $\e\in (0,1)$ let $\mathscr{N}_\e$ be an $\e$-net in $S_X$. Then,
\begin{equation}\label{eq:lower ext with lamda}
\e\le \frac{\kappa^{\frac{1}{\alpha}} M^{1-\alpha}}{\sqrt{m}D \lambda(X,Y)}\implies e^\alpha(Y,\mathscr{N}_\e,\ell_2^m)\gtrsim \frac{\lambda(X,Y)^\alpha}{m^{\frac{1-\alpha}{2}}M^{\alpha(1-\alpha)}D^{2-\alpha}}.
\end{equation}
In particular, there exists $n\in \N$ satisfying $\log n\asymp m\left(\frac{1}{\alpha}+\log m\right)$ for which
\begin{equation}\label{eq:m log m}
\min\left\{e^\alpha(Y,\ell_2^m),e_n^\alpha(Y,\ell_2^m)\right\}\gtrsim \frac{\lambda(X,Y)^\alpha}{m^{\frac{1-\alpha}{2}}M^{\alpha(1-\alpha)}D^{2-\alpha}}.
\end{equation}
\end{corollary}

\begin{proof} Since the Banach--Mazur distance of $X$ to $\ell_2^m$ is at most $D$, the lower bound on $e^\alpha(Y,\mathscr{N}_\e,\ell_2^m)$ that appears in~\eqref{eq:lower ext with lamda} will follow if we show that
\begin{equation}\label{eq:X valued}
e^\alpha(Y,\mathscr{N}_\e,X)\gtrsim \frac{\lambda(X,Y)^\alpha}{m^{\frac{1-\alpha}{2}}M^{\alpha(1-\alpha)}D^{1-\alpha}}.
\end{equation}
We may assume from now on that
\begin{equation}\label{eq:X valued contrapositive}
e^\alpha(Y,\mathscr{N}_\e,X)\le  \frac{\lambda(X,Y)^\alpha}{m^{\frac{1-\alpha}{2}}M^{\alpha(1-\alpha)}D^{1-\alpha}},
\end{equation}
since otherwise there is nothing to prove. By the assumption on $\e$ in~\eqref{eq:lower ext with lamda} combined with~\eqref{eq:X valued contrapositive},
$$
\e\le \frac{\kappa^{\frac{1}{\alpha}} M^{1-\alpha}}{\sqrt{m}D \lambda(X,Y)}\le
\left(\frac{\kappa}{\sqrt{m}D e^\alpha(Y,\mathscr{N}_\e,Z)}\right)^{\frac{1}{\alpha}}.
$$
We can therefore apply Theorem~\ref{thm:linearization} with $Z=X$ and $U=I_X$ being the identity on $X$ (in particular, $C=1$). The operator $V:Y\to X$ thus obtained is a projection onto $X$, since $V|_X=I_X$. Hence,
\begin{multline*}
\lambda(X,Y)\le \|V\|_{Y\to X}\stackrel{\eqref{eq:V conclusion}}{\lesssim} \alpha M^{1-\alpha}m^{\frac{1-\alpha}{2\alpha}}D^{\frac{1-\alpha}{\alpha}}  e^\alpha(Y,\mathscr{N}_\e,X)^{\frac{1}{\alpha}}+\frac{e^\alpha(Y,\mathscr{N}_\e,X)}{\alpha}\\
\lesssim \frac{M^{1-\alpha}m^{\frac{1-\alpha}{2\alpha}}D^{\frac{1-\alpha}{\alpha}}  e^\alpha(Y,\mathscr{N}_\e,X)^{\frac{1}{\alpha}}}{\alpha},
\end{multline*}
which simplifies to give the desired estimate~\eqref{eq:X valued}.

To deduce~\eqref{eq:m log m}, note that by the Kadec$'$--Snobar theorem~\cite{KS71} we have $\lambda(X,Y)\le \sqrt{m}$ and by John's theorem~\cite{Joh48} we have $D\le \sqrt{m}$. Therefore, if we choose $\e=\kappa^{1/\alpha}/m^{3/2}$ then the upper bound on $\e$ that appears in~\eqref{eq:lower ext with lamda} holds true. It remains to note that if we set $n=|\mathscr{N}_\e|$ then $\min\{e^\alpha(X,\ell_2^m),e_n^\alpha(X,\ell_2^m)\}\ge e^\alpha(Y,\mathscr{N}_\e,X)$, and by standard estimates (e.g.~\cite{MS86}) on the size of $\e$-nets in $m$-dimensional normed spaces we have $\log n\asymp m\log (1/\e)\asymp m(\log m+1/\alpha)$.
\end{proof}

The following corollary implies~\eqref{eq:second part of holder}, thus completing the proof of Theorem~\ref{thm:holder}.

\begin{corollary}
For arbitrarily large $m,n\in \N$ and every $\alpha\in (1/2,1]$ we have
$$
e^\alpha(\ell_1,\ell_2^m)\gtrsim m^{\alpha^2-\frac12}\qquad\mathrm{and}\qquad e_n^\alpha(\ell_1,\ell_2)\gtrsim \left(\frac{\log n}{\log\log n}\right)^{\alpha^2-\frac12}.
$$
\end{corollary}

\begin{proof}
By~\cite{FLM77,Kas77}, for every $m\in \N$ there exists an integer $M=O(m)$ and an $m$-dimensional subspace $X$ of $\ell_1^{M}$ whose Banach--Mazur distance to $\ell_2^m$ is $O(1)$. By~\cite{Rut65}, we have $\lambda(X,\ell_1^{M})\gtrsim \sqrt{m}$. Therefore, an application of Corollary~\ref{coro:projection constant} with $D=O(1)$, and $M=O(m)$ shows that there exists $n\in \N$ with $\log n\asymp m\log m$ such that
\begin{equation*}
\min\left\{e^\alpha(\ell_1^M,\ell_2^m),e_n^\alpha(\ell_1^M,\ell_2^m)\right\}\gtrsim \frac{m^{\frac{\alpha}{2}}}{m^{\frac{1-\alpha}{2}+\alpha(1-\alpha)}}=m^{\alpha^2-\frac12}\asymp \left(\frac{\log n}{\log\log n}\right)^{\alpha^2-\frac12}.\tag*{\qedhere}
\end{equation*}
\end{proof}

\bibliographystyle{alphaabbrvprelim}
\bibliography{sqrtlog}

\newcommand{\etalchar}[1]{$^{#1}$}
\def\cprime{$'$} \def\cprime{$'$} \def\cprime{$'$}
\begin{thebibliography}{KKMR09}
\expandafter\ifx\csname urlstyle\endcsname\relax
  \providecommand{\doi}[1]{doi:\discretionary{}{}{}#1}\else
  \providecommand{\doi}{doi:\discretionary{}{}{}\begingroup
  \urlstyle{rm}\Url}\fi

\bibitem[Ach03]{Ach03}
D.~Achlioptas.
\newblock Database-friendly random projections: {J}ohnson-{L}indenstrauss with
  binary coins.
\newblock \emph{J. Comput. System Sci.}, 66(4):671--687, 2003.
\newblock Special issue on PODS 2001 (Santa Barbara, CA).

\bibitem[AFH{\etalchar{+}}04]{AFHKTT04}
A.~Archer, J.~Fakcharoenphol, C.~Harrelson, R.~Krauthgamer, K.~Talwar, and
  {\'E}.~Tardos.
\newblock Approximate classification via earthmover metrics.
\newblock In \emph{Proceedings of the {F}ifteenth {A}nnual {ACM}-{SIAM}
  {S}ymposium on {D}iscrete {A}lgorithms}, pages 1079--1087 (electronic). ACM,
  New York, 2004.

\bibitem[Bal92]{Bal92}
K.~Ball.
\newblock Markov chains, {R}iesz transforms and {L}ipschitz maps.
\newblock \emph{Geom. Funct. Anal.}, 2(2):137--172, 1992.

\bibitem[Bal13]{Bal13}
K.~Ball.
\newblock The {R}ibe programme.
\newblock \emph{Ast\'erisque}, (352):Exp. No. 1047, viii, 147--159, 2013.
\newblock S{\'e}minaire Bourbaki. Vol. 2011/2012. Expos{\'e}s 1043--1058.

\bibitem[Beg99]{Beg99}
B.~Begun.
\newblock A remark on almost extensions of {L}ipschitz functions.
\newblock \emph{Israel J. Math.}, 109:151--155, 1999.

\bibitem[BFdlV82]{Bol82}
B.~Bollob{\'a}s and W.~Fernandez de~la Vega.
\newblock The diameter of random regular graphs.
\newblock \emph{Combinatorica}, 2(2):125--134, 1982.

\bibitem[BL00]{BL00}
Y.~Benyamini and J.~Lindenstrauss.
\newblock \emph{Geometric nonlinear functional analysis. {V}ol. 1}, volume~48
  of \emph{American Mathematical Society Colloquium Publications}.
\newblock American Mathematical Society, Providence, RI, 2000.
\newblock ISBN 0-8218-0835-4.

\bibitem[Bou81]{Bou81}
J.~Bourgain.
\newblock A counterexample to a complementation problem.
\newblock \emph{Compositio Math.}, 43(1):133--144, 1981.

\bibitem[Bou87]{Bou87}
J.~Bourgain.
\newblock Remarks on the extension of {L}ipschitz maps defined on discrete sets
  and uniform homeomorphisms.
\newblock In \emph{Geometrical aspects of functional analysis (1985/86)},
  volume 1267 of \emph{Lecture Notes in Math.}, pages 157--167. Springer,
  Berlin, 1987.
\newblock \doi{10.1007/BFb0078143}.

\bibitem[CKNZ05]{CKNZ04}
C.~Chekuri, S.~Khanna, J.~Naor, and L.~Zosin.
\newblock A linear programming formulation and approximation algorithms for the
  metric labeling problem.
\newblock \emph{SIAM J. Discrete Math.}, 18(3):608--625, 2004/05.

\bibitem[CKR05]{CKR04}
G.~Calinescu, H.~Karloff, and Y.~Rabani.
\newblock Approximation algorithms for the 0-extension problem.
\newblock \emph{SIAM J. Comput.}, 34(2):358--372, 2004/05.

\bibitem[Die10]{Die10}
R.~Diestel.
\newblock \emph{Graph theory}, volume 173 of \emph{Graduate Texts in
  Mathematics}.
\newblock Springer, Heidelberg, fourth edition, 2010.
\newblock ISBN 978-3-642-14278-9.
\newblock \doi{10.1007/978-3-642-14279-6}.

\bibitem[DJP{\etalchar{+}}94]{DJPSY94}
E.~Dahlhaus, D.~S. Johnson, C.~H. Papadimitriou, P.~D. Seymour, and
  M.~Yannakakis.
\newblock The complexity of multiterminal cuts.
\newblock \emph{SIAM J. Comput.}, 23(4):864--894, 1994.

\bibitem[Dvo61]{Dvo60}
A.~Dvoretzky.
\newblock Some results on convex bodies and {B}anach spaces.
\newblock In \emph{Proc. {I}nternat. {S}ympos. {L}inear {S}paces ({J}erusalem,
  1960)}, pages 123--160. Jerusalem Academic Press, Jerusalem; Pergamon,
  Oxford, 1961.

\bibitem[Enf69]{Enf69}
P.~Enflo.
\newblock On the nonexistence of uniform homeomorphisms between
  {$L_{p}$}-spaces.
\newblock \emph{Ark. Mat.}, 8:103--105, 1969.

\bibitem[FHRT03]{FHRT03}
J.~Fakcharoenphol, C.~Harrelson, S.~Rao, and K.~Talwar.
\newblock An improved approximation algorithm for the 0-extension problem.
\newblock In \emph{Proceedings of the {F}ourteenth {A}nnual {ACM}-{SIAM}
  {S}ymposium on {D}iscrete {A}lgorithms ({B}altimore, {MD}, 2003)}, pages
  257--265. ACM, New York, 2003.

\bibitem[FJS88]{FJS88}
T.~Figiel, W.~B. Johnson, and G.~Schechtman.
\newblock Factorizations of natural embeddings of {$l^n_p$} into {$L_r$}. {I}.
\newblock \emph{Studia Math.}, 89(1):79--103, 1988.

\bibitem[FLM77]{FLM77}
T.~Figiel, J.~Lindenstrauss, and V.~D. Milman.
\newblock The dimension of almost spherical sections of convex bodies.
\newblock \emph{Acta Math.}, 139(1-2):53--94, 1977.

\bibitem[FTJ79]{FT79}
T.~Figiel and N.~Tomczak-Jaegermann.
\newblock Projections onto {H}ilbertian subspaces of {B}anach spaces.
\newblock \emph{Israel J. Math.}, 33(2):155--171, 1979.

\bibitem[HLW06]{HLW06}
S.~Hoory, N.~Linial, and A.~Wigderson.
\newblock Expander graphs and their applications.
\newblock \emph{Bull. Amer. Math. Soc. (N.S.)}, 43(4):439--561 (electronic),
  2006.

\bibitem[HW71]{HW71}
T.~L. Hayden and J.~H. Wells.
\newblock On the extension of {L}ipschitz-{H}\"older maps of order {$\beta $}.
\newblock \emph{J. Math. Anal. Appl.}, 33:627--640, 1971.

\bibitem[JL84]{JL84}
W.~B. Johnson and J.~Lindenstrauss.
\newblock Extensions of {L}ipschitz mappings into a {H}ilbert space.
\newblock In \emph{Conference in modern analysis and probability ({N}ew
  {H}aven, {C}onn., 1982)}, volume~26 of \emph{Contemp. Math.}, pages 189--206.
  Amer. Math. Soc., Providence, RI, 1984.

\bibitem[JLS86]{JLS86}
W.~B. Johnson, J.~Lindenstrauss, and G.~Schechtman.
\newblock Extensions of {L}ipschitz maps into {B}anach spaces.
\newblock \emph{Israel J. Math.}, 54(2):129--138, 1986.

\bibitem[Joh48]{Joh48}
F.~John.
\newblock Extremum problems with inequalities as subsidiary conditions.
\newblock In \emph{Studies and {E}ssays {P}resented to {R}. {C}ourant on his
  60th {B}irthday, {J}anuary 8, 1948}, pages 187--204. Interscience Publishers,
  Inc., New York, N. Y., 1948.

\bibitem[Kal04]{Kal04}
N.~J. Kalton.
\newblock Spaces of {L}ipschitz and {H}\"older functions and their
  applications.
\newblock \emph{Collect. Math.}, 55(2):171--217, 2004.

\bibitem[Kal12]{Kal12}
N.~J. Kalton.
\newblock The uniform structure of {B}anach spaces.
\newblock \emph{Math. Ann.}, 354(4):1247--1288, 2012.

\bibitem[Kar98]{Kar98}
A.~V. Karzanov.
\newblock Minimum {$0$}-extensions of graph metrics.
\newblock \emph{European J. Combin.}, 19(1):71--101, 1998.

\bibitem[Ka{\v{s}}77]{Kas77}
B.~S. Ka{\v{s}}in.
\newblock The widths of certain finite-dimensional sets and classes of smooth
  functions.
\newblock \emph{Izv. Akad. Nauk SSSR Ser. Mat.}, 41(2):334--351, 478, 1977.

\bibitem[Kir34]{Kirsz34}
M.~D. Kirszbraun.
\newblock \"{U}ber die zusammenziehenden und {L}ipschitzchen
  {T}ransformationen.
\newblock \emph{Fundam. Math.}, 22:77--108, 1934.

\bibitem[KKMR09]{KKMR09}
H.~Karloff, S.~Khot, A.~Mehta, and Y.~Rabani.
\newblock On earthmover distance, metric labeling, and 0-extension.
\newblock \emph{SIAM J. Comput.}, 39(2):371--387, 2009.

\bibitem[KRTJ80]{KRT80}
H.~K{\"o}nig, J.~R. Retherford, and N.~Tomczak-Jaegermann.
\newblock On the eigenvalues of {$(p,\,2)$}-summing operators and constants
  associated with normed spaces.
\newblock \emph{J. Funct. Anal.}, 37(1):88--126, 1980.

\bibitem[KS71]{KS71}
M.~{\u{I}}. Kadec{\cprime} and M.~G. Snobar.
\newblock Certain functionals on the {M}inkowski compactum.
\newblock \emph{Mat. Zametki}, 10:453--457, 1971.

\bibitem[Lan99]{Lan99}
U.~Lang.
\newblock Extendability of large-scale {L}ipschitz maps.
\newblock \emph{Trans. Amer. Math. Soc.}, 351(10):3975--3988, 1999.

\bibitem[Lew78]{Lew78}
D.~R. Lewis.
\newblock Finite dimensional subspaces of {$L_{p}$}.
\newblock \emph{Studia Math.}, 63(2):207--212, 1978.

\bibitem[Lin64]{Lin64}
J.~Lindenstrauss.
\newblock On nonlinear projections in {B}anach spaces.
\newblock \emph{Michigan Math. J.}, 11:263--287, 1964.

\bibitem[LN03]{LN03}
J.~R. Lee and A.~Naor.
\newblock Metric decomposition, smooth measures, and clustering, 2003.
\newblock Preprint, available on request.

\bibitem[LN04]{LN04}
J.~R. Lee and A.~Naor.
\newblock Absolute {L}ipschitz extendability.
\newblock \emph{C. R. Math. Acad. Sci. Paris}, 338(11):859--862, 2004.

\bibitem[LN05]{LN05}
J.~R. Lee and A.~Naor.
\newblock Extending {L}ipschitz functions via random metric partitions.
\newblock \emph{Invent. Math.}, 160(1):59--95, 2005.

\bibitem[Mat97]{Mat97}
J.~Matou{\v{s}}ek.
\newblock On embedding expanders into {$l_p$} spaces.
\newblock \emph{Israel J. Math.}, 102:189--197, 1997.

\bibitem[Mat02]{Mat02}
J.~Matou{\v{s}}ek.
\newblock \emph{Lectures on discrete geometry}, volume 212 of \emph{Graduate
  Texts in Mathematics}.
\newblock Springer-Verlag, New York, 2002.
\newblock ISBN 0-387-95373-6.
\newblock \doi{10.1007/978-1-4613-0039-7}.

\bibitem[Mau74]{Mau74}
B.~Maurey.
\newblock \emph{Th\'eor\`emes de factorisation pour les op\'erateurs
  lin\'eaires \`a valeurs dans les espaces {$L^{p}$}}.
\newblock Soci\'et\'e Math\'ematique de France, Paris, 1974.
\newblock With an English summary, Ast{\'e}risque, No. 11.

\bibitem[{Men}27]{Men27}
K.~{Menger}.
\newblock {Zur allgemeinen Kurventheorie.}
\newblock \emph{{Fundam. Math.}}, 10:96--115, 1927.

\bibitem[Min70]{Min70}
G.~J. Minty.
\newblock On the extension of {L}ipschitz, {L}ipschitz-{H}\"older continuous,
  and monotone functions.
\newblock \emph{Bull. Amer. Math. Soc.}, 76:334--339, 1970.

\bibitem[MM10]{MM10}
K.~Makarychev and Y.~Makarychev.
\newblock Metric extension operators, vertex sparsifiers and {L}ipschitz
  extendability.
\newblock In \emph{2010 {IEEE} 51st {A}nnual {S}ymposium on {F}oundations of
  {C}omputer {S}cience {FOCS} 2010}, pages 255--264. IEEE Computer Soc., Los
  Alamitos, CA, 2010.

\bibitem[MN13a]{MN13}
M.~Mendel and A.~Naor.
\newblock Spectral calculus and {L}ipschitz extension for barycentric metric
  spaces.
\newblock \emph{Anal. Geom. Metr. Spaces}, 1:163--199, 2013.

\bibitem[MN13b]{MN-skeletons}
M.~Mendel and A.~Naor.
\newblock Ultrametric skeletons.
\newblock \emph{Proc. Natl. Acad. Sci. USA}, 110(48):19256--19262, 2013.

\bibitem[MN15]{MN-duke}
M.~Mendel and A.~Naor.
\newblock Expanders with respect to {H}adamard spaces and random graphs, 2015.
\newblock To appear in Duke Math. J.

\bibitem[MP84]{MP84}
M.~B. Marcus and G.~Pisier.
\newblock Characterizations of almost surely continuous {$p$}-stable random
  {F}ourier series and strongly stationary processes.
\newblock \emph{Acta Math.}, 152(3-4):245--301, 1984.

\bibitem[MS86]{MS86}
V.~D. Milman and G.~Schechtman.
\newblock \emph{Asymptotic theory of finite-dimensional normed spaces}, volume
  1200 of \emph{Lecture Notes in Mathematics}.
\newblock Springer-Verlag, Berlin, 1986.
\newblock ISBN 3-540-16769-2.
\newblock With an appendix by M. Gromov.

\bibitem[Nao01]{Nao01}
A.~Naor.
\newblock A phase transition phenomenon between the isometric and isomorphic
  extension problems for {H}\"older functions between {$L_p$} spaces.
\newblock \emph{Mathematika}, 48(1-2):253--271 (2003), 2001.

\bibitem[Nao12]{Nao12}
A.~Naor.
\newblock An introduction to the {R}ibe program.
\newblock \emph{Jpn. J. Math.}, 7(2):167--233, 2012.

\bibitem[NPSS06]{NPSS06}
A.~Naor, Y.~Peres, O.~Schramm, and S.~Sheffield.
\newblock Markov chains in smooth {B}anach spaces and {G}romov-hyperbolic
  metric spaces.
\newblock \emph{Duke Math. J.}, 134(1):165--197, 2006.

\bibitem[NRS05]{NRS05}
A.~Naor, Y.~Rabani, and A.~Sinclair.
\newblock Quasisymmetric embeddings, the observable diameter, and expansion
  properties of graphs.
\newblock \emph{J. Funct. Anal.}, 227(2):273--303, 2005.

\bibitem[Pis79]{Pis79}
G.~Pisier.
\newblock Estimations des distances \`a un espace euclidien et des constantes
  de projection des espaces de {B}anach de dimension finie;\ d'apr\`es {H}.
  {K}\"onig et al.
\newblock In \emph{S\'eminaire d'{A}nalyse {F}onctionnelle (1978--1979)}, pages
  Exp. No. 10, 21. \'Ecole Polytech., Palaiseau, 1979.

\bibitem[Rut65]{Rut65}
D.~Rutovitz.
\newblock Some parameters associated with finite-dimensional {B}anach spaces.
\newblock \emph{J. London Math. Soc.}, 40:241--255, 1965.

\bibitem[Vil03]{Vil03}
C.~Villani.
\newblock \emph{Topics in optimal transportation}, volume~58 of \emph{Graduate
  Studies in Mathematics}.
\newblock American Mathematical Society, Providence, RI, 2003.
\newblock ISBN 0-8218-3312-X.

\bibitem[Woj91]{Woj91}
P.~Wojtaszczyk.
\newblock \emph{Banach spaces for analysts}, volume~25 of \emph{Cambridge
  Studies in Advanced Mathematics}.
\newblock Cambridge University Press, Cambridge, 1991.

\end{thebibliography}

\end{document}